\documentclass[final]{siamart0216}

\usepackage{lipsum}
\usepackage{amsfonts}
\usepackage{graphicx}
\usepackage{epstopdf}
\usepackage{algorithm,algpseudocode}
\newcommand{\MlM}{\mathcal{M}_{l,m}}
\newcommand{\gge}[1]{}
\newcommand{\Mlr}{\mathcal{M}_{l,r}}
\newcommand{\MMr}{\mathcal{M}_{m,r}}
\newcommand{\Stlr}{\mathrm{St}_{l,r}}
\newcommand{\Mrr}{\mathcal{M}_{r,r}}

\newcommand{\rank}{\mathrm{rank}}
\newcommand{\M}{\mathscr{M}}
\newcommand{\RR}{\mathfrak{R}}

\newcommand{\HH}{\mathcal{H}_{(U,Z)}}
\usepackage{bm}
\usepackage{mathrsfs}

\newcommand{\R}{\mathbb{R}}

\newcommand{\D}{\ensuremath \mathrm{d}}
\newcommand{\<}{\ensuremath \leq}
\renewcommand{\>}{\ensuremath \geq}
\newcommand{\lb}[1]{\label{eqn:#1}}

\newcommand{\XX}{\mathfrak{X}}
\newcommand{\TT}[1]{{\mathcal{T}(#1)}}
\newcommand{\Sk}{\textrm{Sk}}
\newcommand{\NN}[1]{{\mathcal{N}(#1)}}

\newcommand{\LL}{\mathcal{L}}
\newsiamremark{remark}{Remark}
\newcommand{\RN}[1]{%
	\textup{\uppercase\expandafter{\romannumeral#1}}%
}
\newcommand{\Tr}{\mathrm{Tr}}\newcommand{\DD}{\mathrm{D}}
\usepackage{subcaption}
\renewcommand{\(}{\begin{equation}}
\renewcommand{\)}{\end{equation}}
\ifpdf
  \DeclareGraphicsExtensions{.eps,.pdf,.png,.jpg}
\else
  \DeclareGraphicsExtensions{.eps}
\fi

\newcommand{\TheTitle}{A geometric approach to \\ dynamical model--order reduction} 
\newcommand{\TheAuthors}{F. Feppon and P.F.J. Lermusiaux}

\headers{A geometric approach to dynamical model--order reduction}{\TheAuthors}

\title{{\TheTitle}}

\author{
  Florian Feppon
  \and
  Pierre F.J.\,Lermusiaux\thanks{MSEAS, Massachusetts Institute of Technology (\email{feppon@mit.edu}, \email{pierrel@mit.edu}).}
}

\usepackage{amsopn}

\usepackage{cleveref}

\ifpdf
\hypersetup{
	pdftitle={\TheTitle},
	pdfauthor={\TheAuthors}
}
\fi

\begin{document}

	\maketitle

	\begin{abstract} Any model order reduced dynamical system that evolves a modal
		decomposition to approximate the discretized solution of a stochastic PDE can be
		related to a vector field tangent to the manifold of fixed rank matrices. The
		Dynamically Orthogonal (DO) approximation is the canonical reduced order model
		for which the corresponding vector field is the orthogonal projection of the
		original system dynamics onto the tangent spaces of this manifold. The embedded
		geometry of the fixed rank matrix manifold is thoroughly analyzed.  The curvature
		of the manifold is characterized and related to the smallest singular value
		through the study of the Weingarten map.  Differentiability results for the
		orthogonal projection onto embedded manifolds are reviewed and used to derive an
		explicit dynamical system for tracking the truncated Singular Value Decomposition
		(SVD)  of a time-dependent matrix. It is demonstrated that the error made by the
		DO approximation remains controlled under the minimal condition that the original
		solution stays close to the low rank manifold, which translates into an explicit
		dependence of this error on the gap between singular values.  The DO
		approximation is also justified as the dynamical system that applies
		instantaneously the SVD truncation to optimally constrain the rank of the reduced
		solution.  Riemannian matrix optimization is investigated in this extrinsic
		framework to provide algorithms that adaptively update the best low rank
		approximation of a smoothly varying matrix.  The related gradient flow provides a
		dynamical system that converges to the truncated SVD of an input matrix for
		almost every initial data.
	\end{abstract}

	\begin{keywords}
		Model order reduction, fixed rank matrix manifold,  low rank approximation, Singular Value Decomposition, orthogonal projection, curvature, Weingarten map, Dynamically Orthogonal approximation, Riemannian matrix optimization.
	\end{keywords}

	\begin{AMS}
		65C20,  53B21, 65F30,  15A23, 53A07, 35R60, 65M15
	\end{AMS}

	\section{Introduction}  Finding efficient model order reduction methods is an issue
	commonly encountered in a wide variety of domains involving intensive computations
	and expensive high-fidelity simulations
	\cite{schilders2008model,Quarteroni_Rozza_ROM_Springer2013,Kutz_2013,constantine2015_activesubspaces}.
	Such domains include uncertainty quantification
	\cite{ghanem_spanos_2003,lermusiaux_et_al_2006b,smith_UQ_SIAM2013,sullivan_UQ_Springer2015},
	dynamical systems analysis \cite{holmes_etal_PODturbulence_Cambridge1998,
	benner_etal_SIAMreview2015,williams_etal_JNS2015}, electrical engineering
	\cite{Fortuna_etal_Springer2013,Bartel_etal_Springer2014}, mechanical engineering
	\cite{Noack_etal_ROM_Springer2011}, ocean and weather predictions
	\cite{lermusiaux2001evolving,majda_timofeyev_JAS2003,cao_etal_IJNMF2007,rozier_etal_SIAMreview2007},
	chemistry \cite{okino_mavrovouniotis_CR1998}, and biology \cite{lee_etal_JPC2002},
	to name a few.  The computational costs and challenges arise from the complexity of
	the mathematical models as well as from the needs of representing variations of
	parameter values and the dominant uncertainties involved. For example, to quantify
	uncertainties of dynamical system fields, one often needs to solve stochastic
	partial differential equations (PDEs),
	\begin{equation} \label{eqn:SPDE}\partial_t \bm u= \mathscr{L}(t,\bm u;\omega)\, ,
	\end{equation}
	where $t$ is time, $\bm u$ the uncertain dynamical fields, $\mathscr{L}$ a differential operator, and $\omega$ a random event.
	For deterministic but parametric dynamical systems, $\omega$ may represent a large set of possible parameter values that need to be accounted for by the model-order reduction.
	Generally, after both spatial and stochastic/parametric event discretization of the PDE \cref{eqn:SPDE}, or more directly if the focus is on solving a complex system of ordinary differential equations (ODEs), one is interested in the numerical solution of a large system of ODEs of the form
	\begin{equation} \label{eqn:dRstar}\dot{\RR}=\LL (t,\RR),\end{equation}
	where $\LL$ is an operator acting on the space of $l$-by-$m$ matrices $\RR$. In the
	case of a direct Monte-Carlo approach for the resolution of the stochastic PDE
	\cref{eqn:SPDE}, $\LL$ is thought as being the discretization of the differential
	operator $\mathscr{L}$ by using $l$ spatial nodes and $m$ Monte-Carlo realizations
	or parameter values being considered.  Accurate quantification of the
	statistical/parametric properties of the original solution $\bm u$ often require to
	solve such system  \cref{eqn:dRstar} with both a high spatial resolution, $l$, and
	high number of realizations, $m$. Hence, solving \cref{eqn:dRstar} directly with a
	Monte-Carlo approach becomes quickly intractable for realistic, real-time
	applications such as ocean and weather predictions
	\cite{palmer_Rep_Prog_phys2000,lermusiaux_JCP2006} or real-time control
	\cite{Lall_etal_IJRNC2002,Lin_McLaughlin_SIAMJSC2014}.

	A method to address this challenge is to assume the existence of an approximation
	$\bm u_{\textrm{DO}}$ of the solution  $\bm u$ onto a finite number of $r$ spatial
	modes, $\bm u_i(t,x)$, and stochastic coefficients, $\zeta_i(t,\omega)$ (here
	assumed to be both time-dependent
	\cite{lermusiaux_JCP2006,sapsis2009dynamically}),
	\begin{equation} \label{eqn:KLdecomposition} \bm u(t,\bm x;\omega)\simeq {\bm
	u}_\textrm{DO}=\sum_{i=1}^r \zeta_i(t,\omega)\, \bm u_i(t,\bm x),\end{equation} and
	look for a dynamical system that would most accurately govern the evolution of
	these dominant modes and coefficients. The optimal approximation (in the sense that
	the $L^2$ error $\mathbb{E}[||\bm u-\bm u_{\textrm{DO}}||^2]^{1/2}$ is minimized)
	is achieved by the  Karuhnen-Lo\`{e}ve (KL) decomposition
	\cite{Loeve1978,holmes_etal_PODturbulence_Cambridge1998}, whose first $r$ modes
	yields an optimal orthonormal basis $(\bm u_i)$. Many methods, such as polynomial
	chaos expansions \cite{xiu_karniadakis_PCE_SIAM2002}, Fourier decomposition
	\cite{WillcoxMegretski2005}, or Proper Orthogonal Decomposition
	\cite{holmes_etal_PODturbulence_Cambridge1998} rely on the choice of a predefined,
	time-independent orthonormal basis either for the modes, $(\bm u_i)$, or the
	coefficients, $(\zeta_i)$, and obtain equations for the respective unknown
	coefficients or modes by Galerkin projection
	\cite{PetterssonIaccarinoNordstroem2015}. However, the use of modes and
	coefficients that are simultaneously dynamic has been shown to be efficient
	\cite{lermusiaux_JCP2006,lermusiaux_PhysD2007}. Dynamically Orthogonal (DO) field
	equations \cite{sapsis2009dynamically,sapsis2012dynamical} were thus derived to
	evolve adaptively this decomposition for a general differential operator
	$\mathscr{L}$ and allowed to obtain efficient simulations of  stochastic
	Navier-Stokes equations  \cite{ueckermann_et_al_JCP2013}.

	At the discrete level,  the decomposition \cref{eqn:KLdecomposition} is written
	$\RR\simeq R=UZ^T$ where $R$ is a rank $r$ approximation of the full rank matrix
	$\RR$, decomposed as the product of a $l$-by-$r$ matrix $U$ containing the
	discretization of the basis functions, $(\bm u_i)$, and of a $m$-by-$r$ matrix $Z$
	containing the realizations of the stochastic coefficients, $(\zeta_i)$. It is well
	known that such approximation  is optimal (in the Frobenius norm) when $R=UZ^T$  is
	obtained by truncating the Singular Value Decomposition (SVD), \emph{i.e.}\;by
	selecting $U$ to be the singular vectors associated with the $r$ largest singular
	values of $\RR$ and setting $Z=\RR^TU$ \cite{horn_johnson_1985,horn_johnson_1991}.
	In 2007, Koch and Lubich \cite{koch2007dynamical}  proposed a method inspired from
	the Dirac Frenkel variational  principle in quantum physics, to evolve dynamically
	a rank $r$ matrix $R=UZ^T$ that approximates the full dynamical system
	\cref{eqn:dRstar}. The main principle of the method lies in the intuition that one
	can update optimally the low-rank approximation $R$ by projecting
	$\mathcal{L}(t,R)$ onto the tangent space of the manifold constituted by low rank
	matrices. Recently, Musharbash \cite{musharbash2015error} noticed the parallel with
	the DO method, and applied the results obtained in \cite{koch2007dynamical} to
	analyze the error committed by the DO approximation for a stochastic heat equation.
	In fact, in the same way the KL expansion is the continuous analogous of the SVD,
	the discretization of the DO decomposition \cite{sapsis2009dynamically} is strictly equivalent to the
	dynamical low rank approximation of Koch and Lubich \cite{koch2007dynamical} when
	the discretization reduces to simulate the matrix dynamical system
	\cref{eqn:dRstar} of $m$ realizations spatially resolved with $l$ nodes.

	Simultaneously, new approaches have emerged since the 1990s in optimization onto
	matrix sets \cite{edelman1998,absil2009optimization}. The application of Riemannian
	geometry to manifolds of matrices has allowed the development of new optimization
	algorithms, that are evolving orthogonality constraints geometrically rather than
	using more classical techniques, such as Lagrange multiplier methods
	\cite{edelman1998}. Matrix dynamical systems that continuously perform matrices
	operations, such as inversion, eigen- or SVD-decompositions, steepest descents, and
	gradient optimization have thus been proposed
	\cite{brockett1988,dehaene1995continuous,Smith1991}. These continuous--time systems were extended and applied to adaptive uncertainty predictions, learning of dominant
	subspace, and data assimilation \cite{Lermusiaux1997,lermusiaux_DAO1999}.

	The purpose of this article is to extend the analysis and the understanding of the
	DO method in the matrix framework as initiated by \cite{koch2007dynamical} and in the above works, by
	furthering its relation to the Singular Value Decomposition and its geometric
	interpretation as a constrained dynamics on the manifold $\M$ of fixed rank
	matrices.  In the vein of \cite{edelman1998,absil2009optimization,mishra2014}, this
	article utilizes the point of view of differential geometry.  To provide a visual
	intuition, a 3D projection of two 2-dimensional subsurfaces of the manifold $\M$ of
	rank one 2-by-2 matrices is visible on  \cref{fig:plotManifold}. This figure has
	been obtained by using the parameterization  \[
		R(\rho,\theta,\phi)=\rho\begin{pmatrix} \sin(\theta)\sin(\phi) &
			\sin(\theta)\cos(\phi)\\ \cos(\theta)\sin(\phi) &
		\cos(\theta)\cos(\phi)\end{pmatrix}, \; \rho>0, \theta\in
	[0,2\pi],\,\phi\in[0,2\pi],\] on $\M$ and projecting orthogonally two subsurfaces  by
	plotting the first three elements $R_{11}, R_{12}$ and $R_{21}$.  Since the
	multiplication of singular values by a non-zero constant does not affect the rank of
	a matrix, $\M\subset \mathcal{M}_{2,2}$ is a cone, which is consistent with the
	increasing of curvature visible on \cref{fig:plotManifoldSpiral} near the origin.
	More generally, $\M$ is the union of $r$-dimensional affine subspaces of $\MlM$
	supported by the manifold of strictly lower rank matrices. It will actually be proven
	in \cref{sec:curvatureMatrix} that the curvature of $\M$ is inversely proportional to
	the lowest singular value, which diverges as matrices approach a rank strictly less
	than $r$. Hence $\M$ can be understood either as a collection of cones
	(\cref{fig:plotManifoldCone}) or as a multidimensional spiral
	(\cref{fig:plotManifoldSpiral}).
	\begin{figure}
		\centering
		\begin{subfigure}{0.45\linewidth}
			\centering
			\includegraphics[height=4cm]{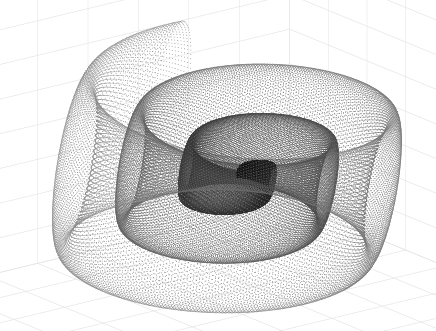}
			\caption{$R(\rho,\pi\rho,\phi)$}
			\label{fig:plotManifoldSpiral}
		\end{subfigure}
		\begin{subfigure}{0.45\linewidth}
			\centering
			\includegraphics[height=4cm]{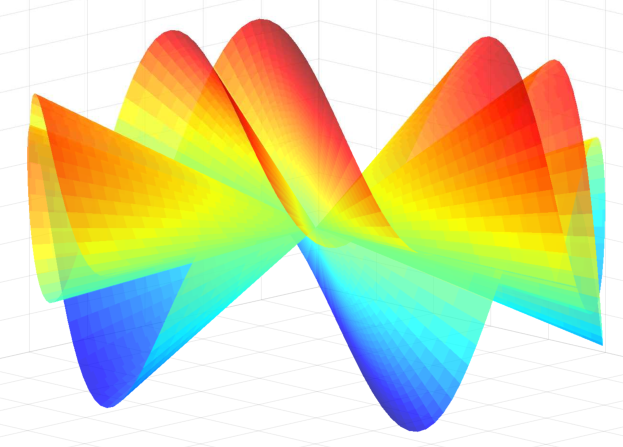}
			\caption{$R(\rho,\phi,\rho-\phi)$}
			\label{fig:plotManifoldCone}
		\end{subfigure}
		\caption{\small Two subsurfaces of the rank-1 manifold $\M$ of 2-by-2 matrices.}
		\label{fig:plotManifold}
	\end{figure}
	Geometrically, a dynamical system \cref{eqn:dRstar} can be seen as a time dependent vector field $\mathcal{L}$ that assigns the velocity $\mathcal{L}(t,\RR)$ at time $t$ at each point $\RR$ of the ambient space $\MlM$ of  $l$-by-$m$ matrices  (\cref{fig:vectorField}). Similarly, any rank $r$ model order reduction can be viewed as a vector field $L$ that must be everywhere tangent to the manifold $\M$ of rank $r$ matrices. The corresponding dynamical system is
	\begin{equation}
		\label{eqn:dRmanifold}
		\dot{R}=L(t,R)\,\in \TT{R},
	\end{equation}
	where $\TT{R}$ denotes the tangent space of $\M$ at $R$.
	\begin{figure}[H]
		\centering
		\begin{subfigure}[t]{0.8\linewidth}
			\includegraphics[width=\linewidth]{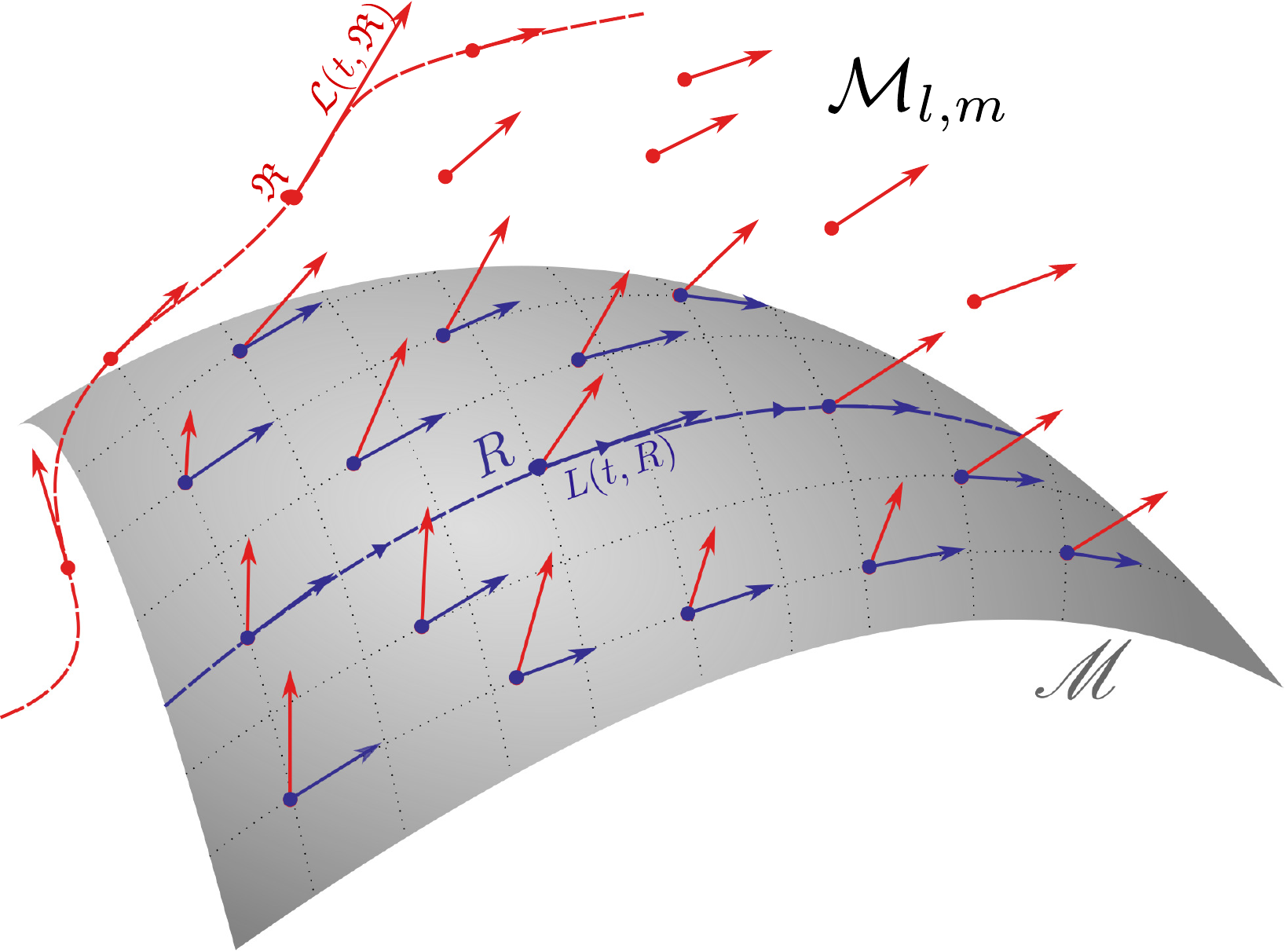}
			\caption{\small Dynamical systems as vector fields $\mathcal{L}$ in the ambient space $\MlM$ (in \emph{red}), or  $L$ tangent to  $\M$ (in \emph{blue}). The DO approximation sets $L$ to be  the  projection of  $\mathcal{L}$ tangent to $\M$.}
			\label{fig:vectorField}
		\end{subfigure}\quad
		\begin{subfigure}[t]{0.8\linewidth}
			\includegraphics[width=\linewidth]{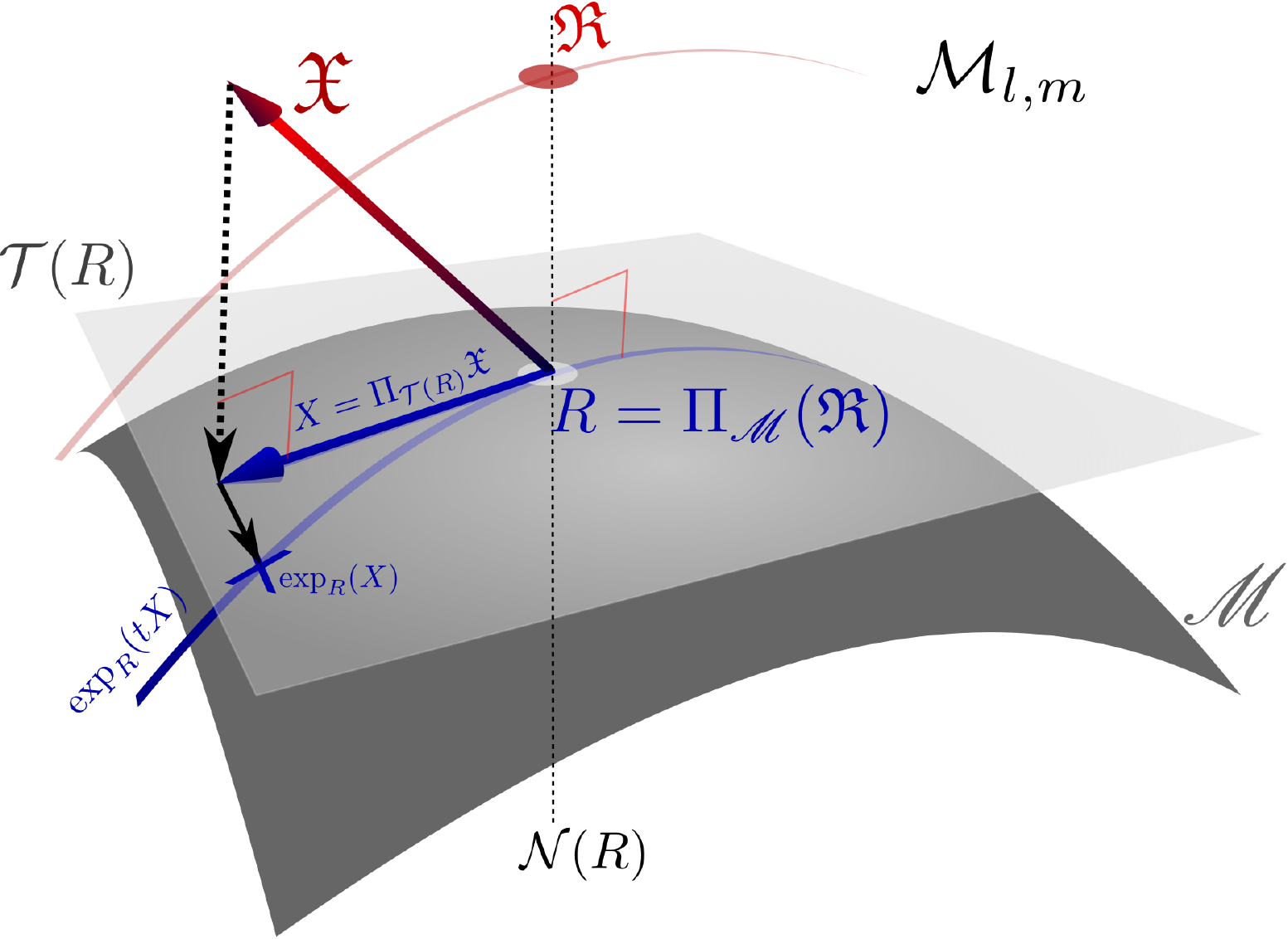}
			\caption{\small Geometric concepts of interest:  orthogonal projection $X=\Pi_\TT{R}\XX$ of an ambient vector $\XX\in\MlM$ onto the tangent space  $\TT{R}$ of $\M$ at $R$.   Orthogonal projection $\Pi_\M(R)$ of the point $\RR$ onto $\M$. Normal space $\NN{R}$. Geodesic curve $\exp_R(tX)$ starting from $R$ in the direction $X$.}
			\label{fig:tangentProj}
		\end{subfigure}
		\caption{\small Vector fields on the fixed rank manifold $\M$. Schematic adapted from \cite{wikiMcSush}.}
		\label{fig:ManifoldFields}
	\end{figure} From this point of view, ``combing the hair'' formed by the original
	vector field $\LL$ on the manifold $\M$, by setting  $L(t,R)$ to the time-dependent
	orthogonal projection of each vector $\XX=\LL(t,R)$ onto each tangent space
	$\TT{R}$ is nothing less than the DO approximation (\cref{fig:tangentProj}). As
	such, the DO-reduced dynamical system is optimal in the sense that the resulting
	vector field $L$ is the best dynamic tangent approximation of $\LL$ at every point
	$R\in\M$.

	Analyzing the error committed by the DO approximation can be done by understanding
	how the best rank $r$ approximation of the solution $\RR$ evolves
	\cite{koch2007dynamical,musharbash2015error}. This requires the time derivative of
	the truncated SVD as a function of $\dot{\RR}$.  Nevertheless, to the best of our knowledge, no explicit expression of the dynamical system satisfied by the
	best low rank approximation has been obtained in the literature. To address this
	gap, this article brings forward the following novelties. First, a more exhaustive
	study of the extrinsic geometry of the fixed rank manifold $\M$ is provided. This includes the
	characterization and derivation of principal curvatures and of their direct relation
	to singular values. Second, the geometric interpretation of the truncated SVD as an
	orthogonal projection onto $\M$ is utilized, so as to apply existing results
	relating the differential of this projection to the curvature of the manifold. It will be
	demonstrated in particular (\cref{thm:diffProjection}) that the truncated SVD is
	differentiable so long as the singular values of order $r$ and $r+1$ remain
	distinct, even if multiple singular values of lower order occur.  As a result, an
	explicit dynamical system is obtained for the evolution of the best low rank
	approximation of the solution $\RR(t)$ of \cref{eqn:dRstar}. This derivation finally also allows a sharpening of the 
	the initial error analysis of \cite{koch2007dynamical}.

	The article is organized as follows: the Riemannian geometric setting is specified
	in \cref{sec:riemanniansetup}. Parameterizations of $\M$ and of its tangent spaces
	are first recalled. Novel geometric characteristics 
	such as covariant derivative and geodesic equations are then derived.  In
	\cref{sec:curvature}, classical results on the differentiability of the orthogonal
	projection onto smooth embedded sub-manifolds \cite{gilbarg2015elliptic} are
	reviewed and reformulated in a framework that avoids the use of tensor notations.
	Curvatures with respect to a normal direction are defined, and their relation to
	the differential of the projection map is stated in \cref{thm:distanceDiff}. These
	results are applied in \cref{sec:curvatureMatrix} where the curvature of the fixed
	rank manifold $\M$ is characterized, and the new formula for the differential of the
	truncated SVD is provided.  The Dynamically Orthogonal approximation (DO) is
	studied in \cref{sec:DO}. Two justifications of the ``reasonable'' character of
	this approximation are given. First, it is shown that this reduced order model
	corresponds to the dynamical system that applies the SVD truncation at all
	instants. The error analysis performed by \cite{koch2007dynamical} is then extended
	and improved using the knowledge of the differential of the truncated SVD. The
	error committed by the DO approximation is shown to be controlled over large
	integration times provided the original solution remains close to the low rank
	manifold $\M$, in the sense that it remains far from the skeleton of $\M$. This
	geometric condition can be expressed as an explicit dependence of the error on the
	gaps between singular values of order $r$ and $r+1$.  Lastly, Riemannian matrix
	optimization on the fixed rank manifold equipped with the extrinsic geometry is
	considered in \cref{sec:numerical} as an alternative approach for tracking the
	truncated SVD. A novel dynamical system is proposed to compute the best low-rank
	approximation, that is shown to be convergent for almost any initial data.

	\subsection*{Notations} Important notations used in this paper are summed up below
	: \begin{table}[H] \centering \scalebox{0.8}{ \begin{tabular}{ll}
			$\mathcal{M}_{l,m}$  & Space of $l$-by-$m$ real matrices \\
			$\mathcal{M}_{m,r}^*$  & Space of $m$-by-$r$ matrices that have full rank\\
			$\textrm{rank}(R)$   & Rank of a matrix $R\in\MlM$\\ $\M=\{R\in
			\MlM|\mathrm{rank}(R)=r\}$ & Fixed rank matrix manifold\\ $\mathcal{O}_r=\{P\in
		\Mrr\;|\;P^TP=I\}$ & Group  of $r$-by-$r$ orthogonal matrices \\
		$\mathrm{St}_{l,r}=\{U\in\Mlr\;|\;U^TU=I\}$  &Stiefel Manifold\\ $R=UZ^T$ & Point
		$R\in\M$ with $U\in \mathrm{St}_{l,r}$ and $Z\in\MMr^*$\\ $\TT{R}$ & Tangent
		space at $R\in\M$\\ $X\in\TT{R}$ & Tangent vector $X$ at $R=UZ^T$\\
		$\mathcal{H}_{(U,Z)}$ & Horizontal space at $R=UZ^T$\\ $(X_U,X_Z)\in
		\mathcal{H}_{(U,Z)}$ & $X=X_UZ^T+UX_Z^T\in \TT{R}$ with \\ &
		$X_U\in\Mlr,\,X_Z\in\MMr$ and  $U^TX_U=0$\\		$\Pi_{\TT{R}}$ & Orthogonal
		projection onto the plane $\TT{R}$\\ $\Sk(\M)$ & Skeleton of $\M$ \\ $\Pi_\M$ &
		Orthogonal projection onto $\M$ (defined on $\MlM\backslash \Sk(\M)$) \\ $I$ &
		Identity  mapping \\ $A^T$  & Transpose of a square matrix $A$ \\ $\langle A,B
		\rangle=\Tr(A^TB)$ & Frobenius matrix scalar product\\ $||A||=\Tr(A^TA)^{1/2}$ &
		Frobenius norm\\		$\sigma_1(A)\>\dots\>\sigma_{\textrm{rank}(A)}(A)$  & Non
		zeros singular values of $A\in\MlM$\\ $\dot{R}=\D R/\D t$ &  Time derivative of a
		trajectory $R(t)$\\ $\DD_X f(R)$ & Differential of a function $f$ in direction
		$X$\\
			$\DD \Pi_{\TT{R}}(X)\cdot Y$ & Differential of the projection operator
			$\Pi_{\TT{R}}$ applied to $Y$\\
		\end{tabular}}
	\end{table} The differential of a smooth function $f$ at the point $R\in \MlM$
	(respectively $R\in\M$) in the direction $X\in\MlM$ (respectively $X\in \TT{R}$) is
	denoted $\DD_X f(R)$: \( \label{eqn:Diff_X_f} \DD_X f(R)=\left.\frac{\D }{\D
			t}f(R(t))\right|_{t=0}=\lim\limits_{\Delta t\rightarrow 0}\frac{f(R(t+\Delta
		t))-f(R(t))}{\Delta t},\) where $R(t)$ is a curve of $\MlM$ (respectively $\M$)
		such that $R(0)=R$ and $\dot{R}(0)=X$. The differential of the orthogonal
		projection operator $R\mapsto \Pi_\TT{R}$ at $R\in\M$, in the direction $X\in
		\TT{R}$ and applied to $Y\in \MlM$ is denoted $\DD \Pi_{\TT{R}}(X)\cdot Y$: \(
			\label{eqn:Diff_PiT_XY} \DD \Pi_{\TT{R}}(X)\cdot Y =  \left[\left.\frac{\D }{\D
					t}\Pi_{\TT{R(t)}}\right|_{t=0}\right](Y)=\left[\lim\limits_{\Delta
			t\rightarrow 0}\frac{\Pi_\TT{R(t+\Delta t)}-\Pi_\TT{R(t)}}{\Delta
	t}\right](Y),\) where $R(t)$ is a curve drawn on $\M$ such that $R(0)=R$  and
	$\dot{R}(0)=X$.

	\section{Riemannian set up: parameterizations, tangent-space, geodesics}
	\label{sec:riemanniansetup}

	This section establishes the geometric framework of low-rank approximation, by
	reviewing and unifying results sparsely available in
	\cite{koch2007dynamical,sapsis2009dynamically,musharbash2015error}, and by
	providing new expressions for classical geometric characteristics, namely
	geodesics and covariant derivative. It is not assumed that the reader is accustomed
	to differential geometry: necessary definitions and properties are recalled.
	Several concepts of this section are illustrated on \cref{fig:tangentProj}.

	\begin{definition} \label{def:manifold} The  manifold of $l$-by-$m$ matrices of
	rank $r$ is denoted by $\M$: \[ \M=\{R\in \MlM|\mathrm{rank}(R)=r\}.\]
\end{definition}
\begin{remark}
	The fact that $\M$ is a manifold is a consequence of the constant rank theorem
	(\cite{spivak1973comprehensive}, Th.10, chap.2, vol. 1) 
	whose assumptions 
	(the map $(U,Z)\mapsto UZ^T$ from $\Stlr\times\MMr^*$ to $\M$ is a submersion with
	differential of constant rank) translate 
	in the requirement that the candidate tangent spaces have constant dimension, as
	found later in \cref{prop:tangentSpace}. Detailed proofs are available in
	\cite{spivak1973comprehensive} (exercise 34, chap. 2, vol. 1) or
	\cite{vandereycken2013low} (Prop. 2.1). 
\end{remark}
The following lemma \cite{Piziak1999} fixes the parametrization of
$\M$ by conveniently representing its elements $R$ in terms of mode and coefficient
matrices, $U$ and $Z$, respectively.

	\begin{lemma}
		\label{lemma:rank}
		Any matrix $R\in\M$ can be decomposed as $R=UZ^T$ where $U\in \Stlr$ and $Z\in
		\MMr^*$,\, \emph{i.e.}\;$U^TU=I \textrm{ and }\mathrm{rank}(Z)=r$, respectively.
		Furthermore, this decomposition is unique modulo a rotation matrix $P\in O_r$,
		namely if $U_1,U_2\in \Mlr$, $Z_1,Z_2\in\MMr$, and $U_1^TU_1=U_2^TU_2=I$, then
		\begin{equation}
			\label{eqn:invariance}
			U_1Z_1^T=U_2Z_2^T\Leftrightarrow \exists P\in \mathrm{O}_r, U
			_1=U_2P\textrm{
		and }Z_1=Z_2P.\end{equation}
	\end{lemma}
	In the following, the statement ``let $UZ^T\in\M$''  always implicitly assumes
	$U\in\Mlr$, $Z\in\MMr$, $U^TU=I$, and $\mathrm{rank}(Z)=r$. Other parameterizations
	of $\M$ are possible and give equivalent results \cite{mishra2014}.

	The tangent space $\TT{UZ^T}$ at a point $R=UZ^T$ is the set of all possible
	vectors tangent to smooth curves $R(t)=U(t)Z(t)^T$ drawn on the manifold $\M$.
	Therefore, such tangent vector at $R(0)=UZ^T$ is of the form
	$\dot{R}=\dot{U}Z^T+U\dot{Z}^T$, where $\dot{U}$ and $\dot{Z}$ are the time
	derivatives of the matrices $U(t)$ and $Z(t)$ at time $t=0$. In the following, the
	notations $X_U$, $X_Z$, and $X=X_UZ^T+UX_Z^T$ will be used to denote the tangent
	directions  $\dot{U}$, $\dot{Z}$, and $\dot{R}$ for the respective matrices $U$,
	$Z$ and $R$.  The orthogonality condition that $U^TU=I$ must hold for all times
	implies that $X_U$ must satisfy $\dot{U}^TU+U^T\dot{U}=X_U^TU+U^TX_U=0$.

	Nevertheless, this is not sufficient to parameterize uniquely tangent vectors $X$
	from the displacements $X_U$ and $X_Z$ for $U$ and $Z$: two different couples
	$(X_U,X_Z)\neq (X_U',X_Z')$ satisfying $X_U^TU+U^TX_U=X_U^{'T}U+U^TX_U'=0$ may
	exist for a single tangent vector $X=X_UZ^T+UX_Z^T=X_U'Z^T+UX_Z^{'T}$. Indeed,
	rotations $U\gets UP$ of the columns of the mode matrix $U$ do not change the
	subspace $\textrm{span}(\bm u_i)$ supporting the modal decomposition
	\cref{eqn:KLdecomposition}, and hence can be captured by updating the values of the
	coefficients $(\zeta_i)$ contained in the matrix $Z$ with the same rotation $Z\gets
	ZP$.  This translates infinitesimally in the tangent space by the invariance of
	tangent vectors $X=X_UZ^T+UX_Z^T$ under the transformations $X_U\gets X_U+U\Omega$
	and $X_Z\gets X_Z+Z\Omega$ for any skew-symmetric matrix $\Omega=-\Omega^T$. This
	can easily be seen by inserting the transformations into the expression for $X$ or by differentiating the relation $UZ^T=(UP)(ZP)^T$ with $\dot{P}=\Omega P$.  A unique
	parameterization of the tangent space can be obtained by fixing this infinitesimal
	rotation $\Omega$, for example by adding the condition that the reduced subspace
	spanned by the columns of $U$ must dynamically evolve orthogonally to itself, in
	other words by requiring $U^TX_U=0$. This gauge condition has thus been called
	``Dynamically Orthogonal'' condition by \cite{sapsis2009dynamically} and is at the
	origin of the name ``Dynamically Orthogonal approximation'' as further investigated
	in \cref{sec:DO}.

	\begin{proposition}
		\label{prop:tangentSpace}
		The tangent space of $\M$ at $R=UZ^T\in\M$ is the set \( \label{eqn:tangentSpace}
		\TT{UZ^T}=\{X_UZ^T+UX_Z^T\;|\; X_U\in\Mlr, \, X_Z\in\MMr, \, U^TX_U+X_U^TU=0\}.\)
		$\TT{UZ^T}$ is uniquely parameterized by the \emph{horizontal space} \(
		\label{eqn:horizSpace} \mathcal{H}_{(U,Z)}=\{ (X_U,X_Z)\in\Mlr\times
	\MMr\;|\;U^TX_U=0\}, \) that is for any tangent vector $X\in \TT{UZ^T}$, there
	exists a unique $(X_U,X_Z)\in \mathcal{H}_{(U,Z)}$ such that $X=X_UZ^T+UX_Z^T$. As
	a consequence $\M$ is a smooth manifold of dimension
	$\mathrm{dim}(\mathcal{H}_{(U,Z)})=(l+m)r-r^2$.
	\end{proposition}
	\begin{proof}
		(see also \cite{koch2007dynamical,absil2009optimization}) One can always write a
		tangent vector $X$ as 
		\[\begin{aligned}X &=U\dot{Z}^T+\dot{U}Z^T\\
											 &=U(\dot{Z}^T+U^T\dot{U}Z^T)
		+((I-UU^T)\dot{U})Z^T=X_UZ^T+UX_Z^T,\end{aligned}\] for some $\dot{U}\in\Stlr$ and
		$\dot{Z}\in\MMr$ with $X_U=(I-UU^T)\dot{U}Z^T$  satisfying $X_U^TU=0$ and
		$X_Z^T=\dot{Z}^T+U^T\dot{U}Z^T$. This implies
		$\TT{UZ^T}=\{X_UZ^T+UX_Z^T|(X_U,X_Z)\in\mathcal{H}_{(U,Z)}\}$.  Furthermore, if
		$X=UX_Z^T+X_UZ^T$ with $U^TX_U=0$, then the relations $X_Z=X^TU$ and
		$X_U=(I-UU^T)XZ(Z^TZ)^{-1}$ show that $(X_U,X_Z)\in\mathcal{H}_{(U,Z)}$ is
		defined uniquely from $X$.
	\end{proof}
	\begin{remark}
		The denomination ``\emph{horizontal space}'' for the set $\mathcal{H}_{(U,Z)}$
		\cref{eqn:horizSpace} refers to the definition of a non-ambiguous representation
		of the tangent space $\TT{UZ^T}$ \cref{eqn:tangentSpace}.  This notion is
		developed rigorously in the theory of quotient manifolds
		e.g.\;\cite{mishra2014,edelman1998}.
	\end{remark}
	In the following, the notation $X=(X_U,X_Z)$ is used equivalently to denote a
	tangent vector $X=X_UZ^T+UX_Z^T\in \TT{UZ^T}$, where  $U^TX_U=0$ is  implicitly
	assumed.

	A metric is needed to define how distances are measured on the manifold, by
	prescribing a smoothly varying scalar product on each tangent space.  In
	\cite{mishra2014} and others in matrix optimization
	e.g.\;\cite{absil2012projection,vandereycken2013low,chechik2011online}, one uses
	the metric induced by the parametrization of the manifold $\M$: the norm of a
	tangent vector $(X_U,X_Z)\in\HH$ is defined to be
	$||(X_U,X_Z)||^2=||X_U||_{\Stlr}^2+||X_Z||_{\MMr}^2$ where $||\;||_{\Stlr}$ is a
	canonical norm on the Stiefel Manifold (see \cite{edelman1998}) and $||\;
	||_{\MMr}$ is the Frobenius norm on $\MMr$.  In this work, one is rather interested
	in the metric inherited from the ambient full space $\MlM$, since it is the metric
	used to estimate the distance from a matrix $\RR\in\MlM$ to its best $r$-rank
	approximation, namely the error committed by the truncated SVD.
	\begin{definition}	 At each point $UZ^T\in\M$, the metric $g$ on $\M$ is the scalar product acting on the tangent space $\TT{UZ^T}$ that is inherited from the scalar product of $\MlM$~:
		\begin{equation} \label{eqn:metric} \begin{aligned}g((X_{U},X_{Z}),(Y_{U},Y_{Z})) & =\Tr((X_{U}Z^T+UX_Z^{T})^T(Y_{U}Z^T+UY_Z^{T}))
				\\ &=\Tr(Z^TZX_U^{T}Y_U+X_Z^{T}Y_Z).\end{aligned}\end{equation}
	\end{definition}

	A main object of this paper is the orthogonal projection $\Pi_\TT{R}$ onto the
	tangent space $\TT{R}$ at a point $R$ on $\M$. This map projects displacements
	$\XX=\dot{\RR}\in\MlM$ of a matrix $\RR$ of the ambient space $\MlM$ to the tangent
	directions $X=\Pi_\TT{R}\XX\in \TT{R}$.
	\begin{proposition}
		At every point $UZ^T\in\M$, the orthogonal projection $\Pi_{\TT{UZ^T}}$ onto the tangent space $\TT{UZ^T}$ is the application
		\begingroup
		\renewcommand*{\arraystretch}{1.2}
		\begin{equation} \label{eqn:projectionMap}\begin{array}{ccccc}
				\Pi_\TT{UZ^T} &: & \MlM & \rightarrow & \HH \\ & & \XX & \mapsto & ((I-UU^T)\XX Z(Z^TZ)^{-1},\XX^TU).\end{array}\end{equation}
		\endgroup
	\end{proposition}
	\begin{proof} (see also \cite{koch2007dynamical})
		$\Pi_\TT{R} \XX$ is obtained as the unique minimizer of the convex functional $J(X_U,X_Z)=\frac{1}{2}||\XX-X_UZ^T-UX_Z^T||^2$ on the space $\mathcal{H}_{(U,Z)}$.
		The minimizer $(X_U,X_Z)$ is characterized by the vanishing of the gradient of $J$:
		\[ \forall \Delta\in\Mlr,\;  \Delta^TU=0\Rightarrow \frac{\partial J}{\partial X_U}\cdot \Delta=-\langle \XX-X_UZ^T-UX_Z^T,\Delta Z^T \rangle=0,\]
		\[ \forall \Delta\in\MMr, \;\frac{\partial J}{\partial X_Z}\cdot \Delta=-\langle \XX-X_UZ^T-UX_Z^T,U \Delta^T \rangle=0, \]
		yielding respectively $X_U=(I-UU^T)\XX Z(Z^TZ)^{-1}$  and $X_Z=\XX^TU$.
	\end{proof}
	The orthogonal complement of the tangent space $\TT{R}$ is obtained from the identity $(I-\Pi_\TT{UZ^T})\cdot \XX=(I-UU^T)\XX(I-Z(Z^TZ)^{-1}Z^T)$:
	\begin{definition}
		The normal space $\NN{R}$ of $\M$ at $R=UZ^T$ is defined as the orthogonal complement to the tangent space $\TT{R}$. For the fixed rank manifold $\M$:
		\begin{equation} \label{eqn:normalSpaceDef}\begin{aligned}\NN{R}& =\{ N\in\MlM| (I-UU^T)N(I-Z(Z^TZ)^{-1}Z^T)=N\}\\
				&=\{ N\in\MlM\;|\;U^TN=0\textrm{ and }NZ=0\}.\end{aligned}\end{equation}
	\end{definition}
	In model order reduction, a matrix $R=UZ^T\in\M$ is usually a low rank-$r$ approximation of a full rank matrix $\RR\in\MlM$. The following proposition shows that the normal space at $R$, $\NN{R}$, can be understood as the set of all possible completions of the approximation \cref{eqn:KLdecomposition}:
	\begin{proposition}
		\label{prop:normalSpace}
		Let $N$ be a given normal vector $N\in\NN{R}$ at $R=UZ^T\in\M$ and denote
		$k=\rank(N)$.
		Then there exists
		an orthonormal basis of vectors $(u_i)_{1\<i\<l}$ in $\R^l$, an orthonormal basis
		$(v_i)_{1\<i\<m}$ of $\R^m$, and $r+k$ non zero singular values
		$(\sigma_i)_{1\<i\<r+k}$ such that
		\begin{equation}
			\label{eqn:singularVectors}
			UZ^T=\sum_{i=1}^r \sigma_i u_
			iv_i^T \textrm{
		and } N=\sum_{i=1}^k \sigma_{r+i}u_{r+i}v_{r+i}^T.\end{equation}
	\end{proposition}
	\begin{proof}
		Consider $N=U_N\Theta V_N^T$ the SVD decomposition of $N$ \cite{horn_johnson_1991}. Since $U^TN=0$, $r$ columns of $U_N$ are spanned by $U$ and associated with zero singular values of $N$, therefore $u_i$ is obtained from the columns of $U$ for $1\<i\<r$ and from the left singular vectors of $N$ associated with non zero singular values for $r+1\<i\<r+k$, $k \ge 0$. The vectors $v_i$ and $v_{r+j}$ are obtained similarly.
		The singular values $\sigma_i$ are obtained by reunion of the respective $r$ and $k$ non-zeros singular values of $Z$  and $N$.
	\end{proof}
	In differential geometry, one distinguishes the geometric properties that are \emph{intrinsic}, \emph{i.e.}\;that depend only on the metric $g$ defined on the manifold, from the ones that are \emph{extrinsic}, \emph{i.e.}\;that depend on the ambient space in which the manifold $\M$ is defined. The following proposition recalls the link between the extrinsic projection $\Pi_\TT{R}$ and the intrinsic notion of derivation onto a manifold. For embedded manifolds, \emph{i.e.}\;defined as subsets of an ambient space, the covariant derivative at $R\in\M$ is obtained by projecting the usual derivative onto the tangent space $\TT{R}$, and the Christoffel symbol corresponds to the normal component that has been removed \cite{edelman1998}.
	\begin{proposition}
		\label{prop:connection}Let $X$ and $Y$ be two tangent vector fields defined on a neighborhood of $R\in \M$. The covariant derivative $\nabla_X Y$ with respect to the metric inherited from the ambient space is the projection of $\DD_X Y$ onto the tangent space $\TT{R}$:
		\[ \nabla_X Y=\Pi_\TT{R} (\DD_X Y).\]
		The Christoffel symbol $\Gamma(X,Y)$  is defined by the relationship  $\nabla_X Y=\DD_X Y+\Gamma(X,Y)$ and is characterized by the formula
		\[ \Gamma(X,Y)=-(I-\Pi_\TT{R})\DD_X Y=-\DD \Pi_\TT{R}(X)\cdot Y.\]
		The Christoffel symbol is symmetric: $\Gamma(X,Y)=\Gamma(Y,X)$.
	\end{proposition}
	\begin{proof}
		See \cite{spivak1973comprehensive}, Vol.3, Ch.1.
	\end{proof}

	\begin{remark}
		An important feature of this definition is that the Christoffel symbol $\Gamma(X,Y)=-\DD\Pi_\TT{R}(X)\cdot Y$,  depends only on the projection map $\Pi_T$ at the point $R$ and not on neighboring values of the tangent vectors $X,Y$, which  is \emph{a priori} not clear from the equality $\Gamma(X,Y)=-(I-\Pi_\TT{R})\DD_X Y$. The Christoffel symbol $\Gamma(X,Y)$ is computed explicitly for the matrix manifold $\M$ in \cref{rmk:christoffel}.
	\end{remark}

	The covariant derivative allows to obtain equations for the geodesics of the manifold $\M$. These geodesics (\cref{fig:tangentProj}) are the shortest paths among all possible smooth curves drawn on $\M$ joining two points sufficiently close.
	Mathematically, they are curves $R(t)=U(t)Z(t)$ characterized by a velocity $\dot{R}=\dot{U}Z^T+U\dot{Z}^T$ that is stationary under the covariant derivative \cite{spivak1973comprehensive}, \emph{i.e.}\;$\nabla_{\dot{R}} \dot{R}=0$. Since $\DD_{\dot{R}}\dot{R}=\ddot{R}$, this leads to
	\begin{equation}
		\label{eqn:geodesicDef}\nabla_{\dot{R}} \dot{R}=\ddot{R}-\DD \Pi_\TT{R}(\dot{R})\cdot \dot{R}=0.
	\end{equation}

	\begin{theorem}
		Consider $X=(X_U,X_Z)\in \mathcal{H}_{(U,Z)}$ and $Y=(Y_U,Y_Z)\in \mathcal{H}_{(U,Z)}$  two tangent vector fields.
		The covariant derivative $\nabla_{X}Y$ on $\M$ is given by
		\begin{equation}\label{eqn:connection}  \nabla_X Y=(D_X Y_U+UX_U^TY_U+(X_UY_Z^T+Y_UX_Z^T)Z(Z^TZ)^{-1},\, D_XY_Z-ZY_U^TX_U).\end{equation}
		Therefore,  geodesic equations on $\M$ are given by
		\begin{equation}
			\label{eqn:geodesics} \left\{\begin{array}{rl}\ddot{U}+U\dot{U}^T\dot{U}+2\dot{U}\dot{Z}^TZ(Z^TZ)^{-1}= & 0\\
				\ddot{Z}-Z\dot{U}^T\dot{U}= & 0.\end{array}\right.\end{equation}
	\end{theorem}
	\begin{proof}
		Writing $X=X_UZ^T+UX_Z^{T}$ and $Y=Y_UZ^T+UY_Z^{T}$, one obtains:
		\begin{align*}
			D_XY &=D_XY_UZ^T+Y_UX_Z^{T}+X_UY_Z^{T}+UD_XY_Z^T\\ &=D_XY_UZ^T+UD_XY_Z^T+X_UY_Z^{T}+Y_UX_Z^{T}.
		\end{align*}
		Applying the projection $\Pi_{T}(UZ^T)$ using eqn.\,\cref{eqn:projectionMap}, \emph{i.e.}
		\[ \nabla_X Y=\Pi_{(U,Z)}(D_X Y)=((I-UU^T)D_X YZ(Z^TZ)^{-1},D_XY^T U),\]
		yields in the coordinates of the horizontal space:
		\[\lb{connection1} \nabla_{X}Y   =((I-UU^T)D_X (Y_U)+(X_UY_Z^T+Y_UX_Z^T)Z(Z^TZ)^{-1},D_X (Y_Z)+ZD_X(Y_U^T)U).\]
		\cref{eqn:connection} is obtained by differentiating the constraint $U^TY_U=0$ along the direction $X$, \emph{i.e.}\;$X_U^TY_U+U^TD_XY_U=0$,
		and replacing accordingly $U^TD_XY_U$ into the above expression. Since $\DD_{(\dot{U},\dot{Z})}(\dot{U})=\ddot{U}$ and $\DD_{(\dot{U},\dot{Z})}(\dot{Z})=\ddot{Z}$,  $\nabla_{(\dot{U},\dot{Z})}(\dot{U},\dot{Z})=0$ yields eqs.\;\cref{eqn:geodesics}.
	\end{proof}
	\begin{remark} Physically,  a curve $R(t)=U(t)Z(t)^T$ describes a geodesic on $\M$ if and only if its acceleration lies in the normal space at all instants (eqn.\;\cref{eqn:geodesicDef})  \cite{edelman1998,spivak1973comprehensive}.
	\end{remark}
	Geodesics allow to define the exponential map \cite{spivak1973comprehensive}, which indicates how to walk on the manifold from a point $R\in \M$ along a straight direction $X\in \TT{R}$.
	\begin{definition}
		The exponential map $\exp_{UZ^T}$ at $R=UZ^T\in \M$ is the function
		\begin{equation}\label{eqn:expmap}\begin{array}{cccc}\exp_{UZ^T}:\; & \TT{UZ^T} & \rightarrow &   \M\\ & X & \mapsto & R(1),\end{array}\end{equation}
		where $R(1)=U(1)Z(1)^T$ is the value at time 1 of the solution of the geodesic equation \cref{eqn:geodesics} with initial conditions $U(0)Z(0)^T=R$ and $(\dot{U}(0),\dot{Z}(0))=X$. The value of the velocity of the point $R(1)=\exp_{UZ^T}(X)$,
		\begin{equation}\label{eqn:parallelTransport}\tau_{RR(1)} X=\dot{U}(1)Z(1)^T+U(1)\dot{Z}(1)^T,\end{equation}
		is called the parallel transport of $X$ from $R$ to $R(1)$.
	\end{definition}

	\section{Curvature and differentiability of the orthogonal projection onto smooth
	embedded manifolds}
	\label{sec:curvature}
	Differentiability results for the orthogonal projection onto smooth embedded
	manifolds, as presented with tensor notations in \cite{Ambrosio2000}, are now
	centralized and adapted to the present study. The main motivation is that the SVD
	truncation (\cref{sec:curvatureMatrix}) is an example of such orthogonal projection
	in the particular case of the fixed-rank manifold. Hence, general geometric
	differentiability results for the projections will transpose directly into a
	formula for the differential of the application mapping a matrix to its best low
	rank approximation. The same analysis can be applied to other matrix manifolds to
	obtain the differential of other algebraic operations, and even generalized to
	non-Euclidean ambient spaces, which is the object of \cite{Feppon2016b}.  In this
	section, the space of $l$-by-$m$ matrices $\MlM$ is replaced with a general finite
	dimensional Euclidean space $E$, and the fixed rank manifold with any given smooth
	embedded manifold $\M\subset E$.

	\begin{definition}
		\label{def:projectionMap}
		Let $\M$ be a smooth manifold embedded in an Euclidian space $E$. The orthogonal
		projection of a point $\RR$ onto $\M$ is defined whenever there is a unique point
		$\Pi_\M(\RR)\in\M$ minimizing the Euclidean distance from $\RR$ to $\M$,
		\emph{i.e.}
		\[ ||\RR-\Pi_\M(\RR)||=\inf_{R\in\M} ||\RR-R||.\]
	\end{definition}
	A fundamental property of the orthogonal projection
	is that the vector $\RR-R$ is normal to $\M$
	for the  point $R=\Pi_\M(\RR)$,
	as geometrically illustrated on \cref{fig:tangentProj}:
	\begin{proposition}
		\label{prop:normal}
		Whenever $\Pi_\M(\RR)$ is defined, the residual $\RR-\Pi_\M(\RR)\in\NN{R}$
		must be normal to $\M$ at $R$, namely
		\begin{equation} \label{eqn:SVDorthogonality} \Pi_\TT{\Pi_\M(\RR)}(\RR-\Pi_\M(\RR))=0.\end{equation}
	\end{proposition}
	\begin{proof}
		For any tangent vector $X\in \TT{R}$, consider a curve $R(t)$ drawn on $\M$ such that $R(0)=R$ and $\dot{R}(0)=X$ where $R$ is minimizing $J(R)=\frac{1}{2}||\RR-R||^2$. Then the stationarity condition
		$\left.\frac{\D}{\D t}\right|_{t=0}J(R(t))=-\langle \RR-R,X \rangle=0$
		states precisely \cref{eqn:SVDorthogonality}.
	\end{proof}
	The following proposition, also used in the proofs of \cite{koch2007dynamical},
	provides an equation for the differential of $\Pi_{\M}$, that will be solved by the
	study of the curvature of $\M$.
	\begin{proposition}
		Suppose the projection $\Pi_\M$ is defined and differentiable at $\RR$. Then the differential $\DD_\XX\Pi_{\M}(\RR)$ of $\Pi_\M$ at the point $\RR$ in the direction $\XX\in E$ satisfies~:
		\begin{equation}\label{eqn:SVDdifferential1} \DD_\XX \Pi_\M(\RR)=\Pi_\TT{\Pi_\M(\RR)}(\XX)+\DD \Pi_\TT{\Pi_\M(\RR)}(\DD_\XX \Pi_\M(\RR))\cdot (\RR-\Pi_\M(\RR)).\end{equation}
	\end{proposition}
	\begin{proof}
		Differentiating equation \cref{eqn:SVDorthogonality} along the direction $\XX$ yields
		\[ \DD \Pi_\TT{\Pi_\M(\RR)}(\DD\Pi_\M(\RR)(\XX))\cdot (\RR-\Pi_\M(\RR))+\Pi_\TT{\Pi_\M(\RR)}(\XX-\DD_\XX\Pi_\M(\RR))=0.\]
		Since $\Pi_\M(\RR)\in\M$ for any $\RR$, the differential $\DD_\XX\Pi_\M(\RR)$ is a  tangent vector, and the results follows from the relation $\Pi_\TT{\Pi_\M(\RR)}(\DD_\XX\Pi_\M(\RR))=\DD_\XX\Pi_\M(\RR)$.
	\end{proof}
	Let $R=\Pi_\M(\RR)$ be the projection of the point $\RR$ on $\M$ and $N=\RR-R$ the
	corresponding normal residual vector. Solving \cref{eqn:SVDdifferential1} for the
	differential $X=\DD_\XX \Pi_\M(\RR)$ requires to invert the linear operator
	$I-L_R(N)$ where $L_R(N)$ is the map $X\mapsto\DD \Pi_\TT{R}(X)\cdot N$.  $L_R(N)$
	would be zero if $\M$ were to be a ``flat'' vector subspace and can be interpreted
	as a curvature correction. In fact, $L_R(N)$ is nothing else than the Weingarten
	map, at the origin of the definition of 	principal curvatures. For embedded
	hypersurface, this application maps tangent vectors $X$ to the  tangent variations
	$-\DD_X N$ of the unit normal vector field $N$, and the eigenvalues and
	eigenvectors of this symmetric endomorphism define the principal curvatures and
	directions of the hypersurface (\cite{spivak1973comprehensive}, Vol. 2).  For
	general smooth embedded sub-manifolds, a Weingarten map is defined for every
	possible normal direction
	\cite{Simon1983,Ambrosio2000,absil2013extrinsic,absil2009all}.

	\begin{definition}[Weingarten map] \label{def:weingartenMap}
		For any point $R\in\M$, tangent and normal vector fields $X,Y\in \TT{R}$ and
		$N\in\NN{R}$ defined on a neighborhood of $R$, the following relation, called
		\emph{Weingarten identity} holds:
		\begin{equation} \label{eqn:weingartenIdentity}\langle \Pi_\TT{R}(\DD_X N),Y \rangle=\langle N,\Gamma(X,Y) \rangle.\end{equation}
		Also, the tangent variations $\Pi_\TT{R}(\DD_X N)$ depend only on the value of the normal vector field $N$ at $R$ as it can be seen from the identity
		\begin{equation} \label{eqn:weingartenApplication}\DD\Pi_\TT{R}(X)\cdot N=-\Pi_\TT{R}(\DD_X N).\end{equation}
		The application \[\begin{array}{ccccc}L_R(N) &: & \TT{R} & \rightarrow & \TT{R}
		\\ & & X & \mapsto &  \DD\Pi_\TT{R}(X)\cdot N,\end{array}\] is therefore a
		symmetric map of the tangent space into itself and is called the Weingarten map
		in the normal direction $N$. The corresponding eigenvectors and eigenvalues are
		respectively called the \emph{principal directions} and \emph{principal
		curvatures} of $\M$ in the normal direction~$N$. The induced symmetric bilinear
		form on the tangent space,
		\begin{equation}\label{eqn:secondFundamentalFormAbstract}\RN{2}(N):\;(X,Y)\mapsto -\langle N,\Gamma(X,Y) \rangle,\end{equation} is called the second fundamental form in the direction $N$.
	\end{definition}
	\begin{proof} See \cite{Simon1983} or the proof Theorem 5 of \cite{spivak1973comprehensive}, vol.3, ch.1.
	\end{proof}

	The differentiability of the projection map for arbitrary sets  has been studied in
	\cite{wulbert1968continuity,abatzoglou1979metric} and more recently in the context
	of smooth manifolds in
	\cite{Ambrosio2000,gilbarg2015elliptic,cannarsa2004representation} with recent
	applications in shape optimization \cite{allaire2014multi}.  The following theorem
	reformulates these results in the framework of this article.
	The proof given in \cref{app:proofDiff} is essentially a justification that one can
	indeed invert the operator $I-L_R(N)$ by using its eigendecomposition. Recall that
	the adherence $\overline{\M}$ is the set of limit points of $\M$. In this paper,
	the boundary of a manifold is defined as the set   $\partial
	\M=\overline{\M}\backslash\M$ .

	\begin{theorem} \label{thm:distanceDiff}
		Let $\Omega\subset E$ be an open set of $E$ and assume that  for any
	$\RR\in\Omega$, there exists a unique projection $\Pi_\M(\RR)\in \M$ such that
\begin{equation}\label{eqn:CondunicityProj}||\RR-\Pi_\M(\RR)||=\inf_{R\in\M}||\RR-R||,\end{equation}
	and that in addition, there is no other projection on the boundary $\partial \M$ of
	$\M$:
		\begin{equation}\label{eqn:Condfrontier} \forall R\in \overline{\M}\backslash \M,
		||\RR-R||>||\RR-\Pi_\M(\RR)||.\end{equation} For  $\RR\in\Omega$, denote
		$\kappa_i(N)$ and $\Phi_i$ the respective eigenvalues and eigenvectors of the
		Weingarten map $L_R(N)$ at $R=\Pi_\M(\RR)$ with the normal direction
		$N=\RR-\Pi_\M(\RR)$. Then all the principal curvatures satisfy $\kappa_i(N)<1$
		and the projection $\Pi_\M$ is differentiable at $\RR$. The differential
		$\DD_\XX\Pi_\M(\RR)$ at $\RR$ in the direction $\XX$ satisfies
		\begin{equation}\label{eqn:diffProjection} \begin{aligned}
				\DD_\XX \Pi_\M(\RR) &  =\sum\limits_{\kappa_i(N)}
			\frac{1}{1-\kappa_i(N)}\langle \Phi_i,\XX \rangle\Phi_i \\
		&=\Pi_{\TT{\Pi_\M(\RR)}}(\XX)+\sum\limits_{\kappa_i(N)\neq 0}
		\frac{\kappa_i(N)}{1-\kappa_i(N)}\langle \Phi_i,\XX \rangle\Phi_i .\end{aligned}
	\end{equation} \end{theorem}
	\begin{proof} See \cref{app:proofDiff} or \cite{Ambrosio2000}.
	\end{proof}

	The set $\Sk(\M)\subset E$ of points that admit more than one possible projection
	is called the \emph{skeleton} of $\M$ (see \cite{delfour2011shapes}). One cannot
	expect the projection map to be differentiable at points that are in the adherence
	$\overline{\Sk(\M)}$, as there is a ``jump'' of the projected values across
	$\Sk(\M)$ (\cref{fig:parabola}).

	Equation \cref{eqn:diffProjection} is analogous to the formula presented in
	\cite{gilbarg2015elliptic} for hyper-surfaces (Lemma 14.17).  In this framework,
	one retrieves the usual  notion of principal curvature  by considering the
	eigenvalues $\kappa_i(N)$ for a normalized normal vector $N$. Curvature radius
	being defined as  inverse of curvatures:
	$\rho_i=\kappa_i\left(\frac{N}{||N||}\right)^{-1}$, the condition
	$\kappa_i(N)=||\RR-\Pi_\M(\RR)||/\rho_i\neq 1$ states that the projection $\Pi_\M$
	is differentiable at points $\RR$ that are not center of curvature. Note that
	assumption (\ref{eqn:Condfrontier}) is required to deal with non closed manifolds
	(boundary points being  not considered as part of the manifold), which is the case
	for the fixed rank matrix manifold.
	\begin{figure}
		\centering
		\includegraphics[width=0.5\linewidth]{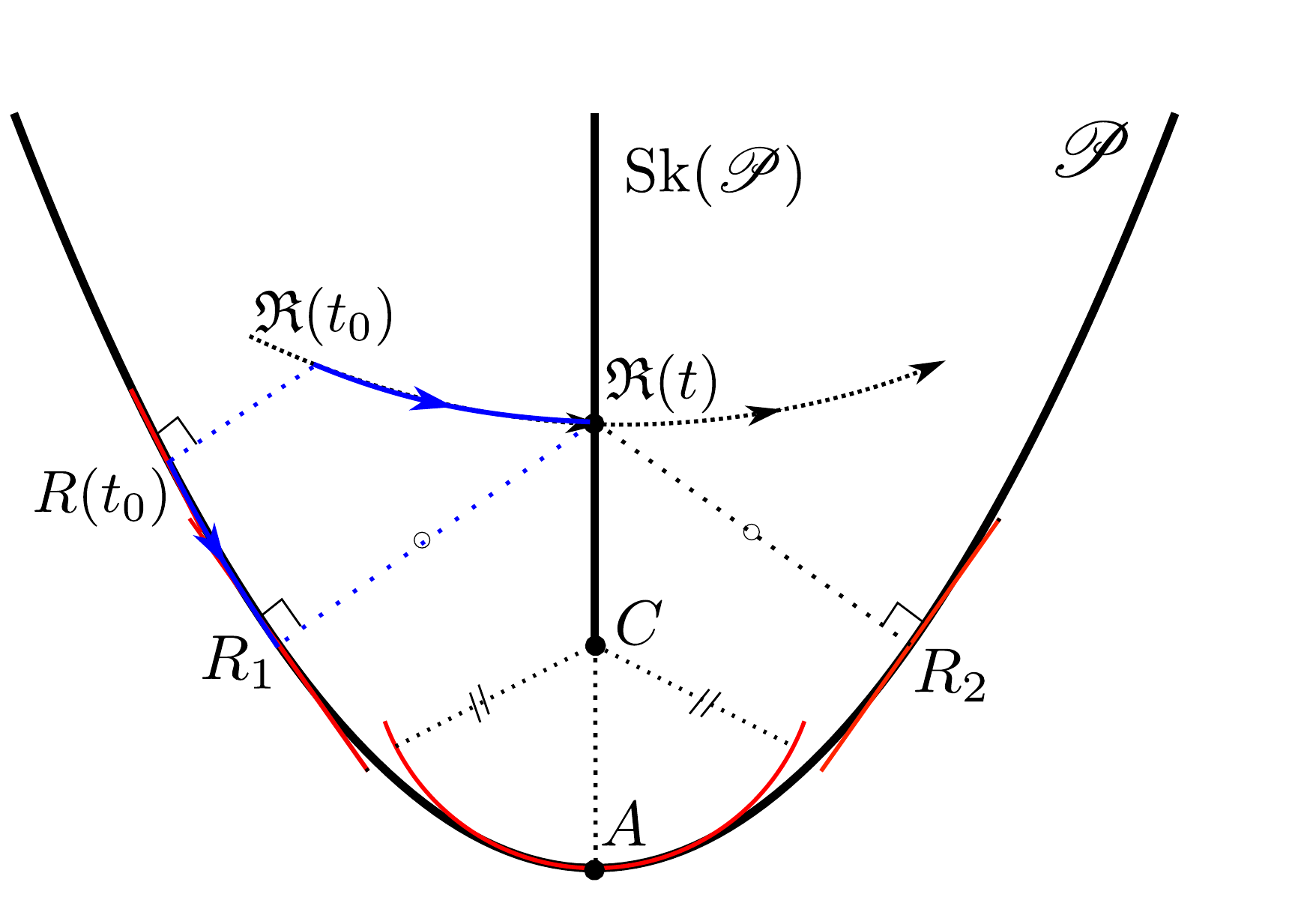}
		\caption{\small A parabola $\M=\mathscr{P}$ and its skeleton set $\Sk(\mathscr{P})$. The orthogonal projection $\Pi_\M$ is not differentiable on the adherence $\overline{\Sk(\mathscr{P})}$. Projected values $R(t)=\Pi_\M(\RR(t))$ jump from $R_1$ to $R_2$ when $\RR(t)$ crosses the skeleton. A center of curvature $C$, for which $\kappa_i(C-\Pi_\M(C))=1$,  may admit a unique projection, $A$, but is a limit point of the skeleton ${\Sk(\mathscr{P})}$.}
		\label{fig:parabola}
	\end{figure}

	\section{Curvature of the fixed rank matrix manifold and the differentiability of the SVD truncation}
	\label{sec:curvatureMatrix}
	In the following, $\M\subset \MlM$ denotes again the fixed rank matrix manifold of
	\cref{def:manifold} and $E=\MlM$ is the space of $l$-by-$m$ matrices. It is well
	known \cite{Golub2012,horn_johnson_1985} that the truncated SVD, \emph{i.e.}\;the
	map that set all singular values of a matrix $\RR$ to zero except the $r$ highest,
	yields the best rank $r$ approximation.
	\begin{definition}
		\label{def:projectionM}
		Let $\RR\in\MlM$ a matrix of rank at least $r$, i.e.\;$r+k, k \ge 0$, and denote
		$\RR=\sum_{i=1}^{r+k}\sigma_i(\RR)u_iv_i^T$ its singular value decomposition. If
		$\sigma_r(\RR)>\sigma_{r+1}(\RR)$, then the rank $r$ truncated SVD \[
		\Pi_{\M}(\RR)=\sum_{i=1}^r \sigma_i(\RR) u_iv_i^T\in\M,\] is the unique matrix
		$R\in\M$ minimizing the Euclidian distance $R\mapsto ||\RR-R||$.
	\end{definition}

	\begin{remark}
		The skeleton of $\M$ (Fig.\;\ref{fig:parabola}) is therefore the set \[\Sk(\M)=\{
		\sigma_r(\RR)=\sigma_{r+1}(\RR)\}\] characterized by the crossing of the singular
		values of order $r$ and $r+1$.
	\end{remark}
	In the following, the Weingarten map for the fixed rank manifold is derived. Note
	that its expression has been previously found by \cite{absil2013extrinsic} under
	the form of equation \cref{eqn:weingartenAbsil} below.
	\begin{proposition}
		\label{prop:matrixWeingarten}
		The Weingarten map $L_R(N)$ of the fixed rank manifold $\M$ in the normal
		direction $N\in \NN{R}$ is the application:
		\begin{equation}\label{eqn:matrixweingartenMap}\begin{array}{crcc}L_R(N)\;:\;&
		\mathcal{H}_{(U,Z)} & \longrightarrow & \mathcal{H}_{(U,Z)} \\ & (X_U,X_Z) &
	\longmapsto & (NX_Z(Z^TZ)^{-1},N^TX_U). \end{array}\end{equation} Or, denoting
	$R=\sum_{i=1}^r \sigma_i u_i v_i^T$ and $N=\sum_{j=1}^k
	\sigma_{r+j}u_{r+j}v_{r+j}^T$ as in \cref{prop:normalSpace}, this can be rewritten
	more explicitly as 
			\begin{equation}
			\label{eqn:weingartenClear}
			\forall X\in\TT{R},\, L_R(N)X=\sum_{\substack{1\<
				i\<r\\1\<j\<k}}
			\frac{\sigma_{r+j}}{\sigma_i}\left[u_iv_i^TX^Tu_{r+j}v_{r+j}^T+u_{r+j}v_{r+j}^TX^Tu_iv_i^T\right].\end{equation}

		The second fundamental form is given by:
		\begin{equation}  \label{eqn:secondFundamentalForm} \RN{2}\;:\;(X,Y)\mapsto
		\langle X,L_R(N)(Y) \rangle=\Tr((X_UY_Z^T+Y_UX_Z^T)^TN).\end{equation}
	\end{proposition}
	\begin{proof}
		Differentiating \cref{eqn:projectionMap} along the tangent direction
		$X=(X_U,X_Z)\in\mathcal{H}_{(U,Z)}$, and using the relations $U^TN=0$ and $NZ=0$,
		yields
		\begin{equation}\label{eqn:weingartenMatrix}L_R(N)X=
		UX_U^TN+NX_Z(Z^TZ)^{-1}Z^T.\end{equation} The normality of $N$ implies that
		$(NX_Z(Z^TZ)^{-1},N^TX_U)$ is a vector of the horizontal space and therefore
		equation \cref{eqn:matrixweingartenMap} follows.  Eqn.
		\cref{eqn:weingartenMatrix} can be rewritten as 
		\begin{equation}
			\label{eqn:weingartenAbsil}
			L_R(N)X=
		U(Z^TZ)^{-1}Z^TX^TN+NX^TU(Z^TZ)^{-1}Z^T,
	\end{equation} by expressing
		$X_U=(I-UU^T)XZ(Z^TZ)^{-1}$ and $X_Z=X^TU$ in terms of $X$ (eqn.
		\cref{eqn:projectionMap}), from which is derived eqn. \cref{eqn:weingartenClear}
		by introducing singular vectors $(u_i)$, $(v_i)$ and singular values
		$(\sigma_i)$.  One obtains\cref{eqn:secondFundamentalForm}   by evaluating the
		scalar product $\langle X,L_R(N)(Y) \rangle$ with the metric $g$ (equation
		\cref{eqn:metric}).
	\end{proof}
	\begin{remark}
		\label{rmk:christoffel}
		The Christoffel symbol is deduced from equations \cref{eqn:secondFundamentalForm}
		and \cref{eqn:secondFundamentalFormAbstract}:
		\begin{equation}
			\label{eqn:christoffel}
			\Gamma(X,Y)=-(I-\Pi_\TT{R})(X_UY_Z^T+Y_UX_Z^T).
		\end{equation}
	\end{remark}

	\begin{theorem}
		\label{thm:curvature}
		Consider a point $R=UZ^T=\sum_{i=1}^r \sigma_i u_iv_i^T\in\M$ and a normal vector
		$N=\sum_{j=1}^k \sigma_{r+j}u_{r+j}v_{r+j}^T\in \NN{R}$ (no ordering of the
		singular values is assumed). At $R$ and in the direction $N$, there are $2kr$
		non-zero principal curvatures
		\[\kappa_{i,j}^\pm(N)=\pm\frac{\sigma_{r+j}}{\sigma_i},\] for all possible
		combinations of non-zero singular values $\sigma_{r+j}, \sigma_i$ for $1\<i\<r$
		and $1\<j\<k$. The normalized corresponding principal directions are the tangent
		vectors  \begin{equation}
			\label{eqn:principalDirectionMatrix}
			\Phi_{i,r+j}^\pm=\frac{1}{\sqrt{2}}(u_{r+j}v_i^T\pm u_iv_{r+j}^T).\end{equation}
		The other principal curvatures are null and associated with  the principal subspace
		\[\mathrm{Ker}(L_R(N))=\mathrm{span}\{(u_iv^T)_{1\<i\<r} | Nv=0\}\oplus\mathrm{span}\{(uv_i^T)_{1\<i\<r} | u^TN=u^TU=0\}.\]
	\end{theorem}
	\begin{proof}
						From \cref{eqn:weingartenClear}, it is clear that
			$L_R(N)\Phi_{i,r+j}^\pm=\kappa_{i,r+j}^\pm(N) \Phi_{i,r+j}^\pm$. In addition,
			$\Phi_{i,r+j}^\pm$ is indeed a tangent vector as one can write
			$\Phi_{i,r+j}^\pm=X_UZ^T\pm UX_Z^T$ with: \[
				(X_U,X_Z)=\frac{1}{\sqrt{2}\sigma_{r+j}\sigma_i}(Nv_{r+j}u_i^TU,
			N^Tu_{r+j}v_i^TZ).\] Therefore $(\Phi_{i,r+j}^\pm)$ is a family of $2kr$
			independent eigenvectors.  Then it is easy to check that
			$\mathrm{span}\{(u_iv^T)_{1\<i\<r} | Nv=0\}$ and
			$\mathrm{span}\{(uv_i^T)_{1\<i\<r} | u^TN=u^TU=0\}$ are null eigenspaces of
			respective dimension $(m-k)r$ and $(l-k-r)r$. The total dimension obtained
			is $(m-k)r+(l-k-r)r+2kr=mr+lr-r^2$, implying that the full spectral
			decomposition has been characterized.
	\end{proof}
	This theorem shows that the maximal curvature of $\M$ (for normalized normal
	directions $||N||=1$) is $\sigma_r(\RR)^{-1}$ and hence diverges as the smallest
	singular value goes to 0. This fact
	confirms what is visible on \cref{fig:plotManifold}:  the manifold $\M$ can be seen
	as a collection of cones or as a multidimensional spiral, whose axes are the lower
	dimensional manifolds of matrices of rank strictly  less than $r$. Applying directly the
	formula \cref{eqn:diffProjection} of \cref{thm:distanceDiff}, one obtains an
	explicit expression for the differential of the truncated SVD:
	\begin{theorem}
		\label{thm:diffProjection}
		Consider $\RR\in\MlM$ with rank greater than $r$ and denote
		$\RR=\sum_{i=1}^{r+k}\sigma_i u_iv_i^T$ its SVD decomposition, where the singular
		values are ordered decreasingly: $\sigma_1\>\sigma_2\>\dots\>\sigma_{r+k}$.
		Suppose that the orthogonal projection  $\Pi_{\M}(\RR)=UZ^T$ of $\RR$ onto $\M$
		is uniquely defined, that is $\sigma_r>\sigma_{r+1}$. Then $\Pi_\M$, the
		truncated SVD of order $r$, is differentiable at $\RR$ and the differential
		$\DD_\XX\Pi(\RR)$ in a direction $\XX\in\MlM$ is given by the  formula
		\begin{multline}\label{eqn:diffSVD}
			\DD_\XX\Pi_{\M}(\RR)=\Pi_{\TT{\Pi_{\M}(R)}}(\XX)\\+\sum_{\substack{1\<i\<r\\1\<j\<k}} \left[\frac{\sigma_{r+j}}{\sigma_i-\sigma_{r+j}}
\langle\XX,\Phi_{i,r+j}^+
\rangle\Phi_{i,r+j}^+-\frac{\sigma_{r+j}}{\sigma_i+\sigma_{r+j}}\langle
\XX,\Phi_{i,r+j}^- \rangle\Phi_{i,r+j}^-\right],\end{multline} where
$\Phi_{i,r+j}^\pm$ are the principal directions of equation
\cref{eqn:principalDirectionMatrix}.  More explicitly,
		\begin{multline}
		\label{eqn:diffSVDexplicit}
	\DD_\XX \Pi_\M(\RR)=(I-UU^T)\XX Z(Z^TZ)^{-1}Z^T+UU^T\XX \\ +\sum_{\substack{1\<i\<r\\1\<j\<k}} \frac{\sigma_{r+j}}{\sigma_i^2-\sigma_{r+j}^2} [ (\sigma_i u_{r+j}^T\XX v_i+\sigma_{r+j} u_i^T\XX v_{r+j})u_{r+j}v_i^T\\+(\sigma_{r+j}u_{r+j}^T\XX v_i+\sigma_iu_i^T\XX v_{r+j})u_iv_{r+j}^T].\end{multline}
	\end{theorem}
	\begin{proof}
		The set $\{\RR\in\MlM, \sigma_{r+1}(\RR)>\sigma_r(\RR)\}$ is open by continuity
		of the singular values, therefore condition (\ref{eqn:CondunicityProj}) of
		 \cref{thm:distanceDiff} is fulfilled. The boundary
		$\overline{\M}\backslash \M$ is the set of matrices of rank strictly lower than
		$r$, hence condition (\ref{eqn:Condfrontier}) is also fulfilled.  Equation
		(\ref{eqn:diffSVD}) follows by replacing $\kappa_i(N)$ and $\Phi_i$ in
		\cref{eqn:diffProjection} by the corresponding curvature eigenvalues
		$\pm\frac{\sigma_{r+j}}{\sigma_i}$ and eigenvectors $\Phi_{i,r+j}^\pm$ of
		\cref{thm:curvature}.
	\end{proof}
	\begin{remark}
	Dehaene  \cite{dehaene1995continuous}
		 and Dieci and Eirola \cite{dieci1999smooth} have previously derived formulas for 
		the time derivative of singular values and singular vectors of a smoothly varying
		matrix. One can also certainly use these results to find 
		formula \cref{eqn:diffSVDexplicit} 
		by differentiating singular values $(\sigma_i)$ and singular vectors 
		$(u_i), (v_i)$ separately in $\sum_{i=1}^r \sigma_i u_i v_i^T$. In the present work, 
		the proof of \cref{thm:diffProjection} does not require singular values to remain
		simple, and formula \cref{eqn:diffSVD} is obtained directly from its geometric
		interpretation.
	\end{remark}
	\section{The Dynamically Orthogonal Approximation} \label{sec:DO}

	The above results are now utilized for model order reduction. Following the
	introduction, the DO approximation is defined to be the dynamical system obtained
	by replacing the vector field $\mathcal{L}(t,\cdot)$ with its tangent projection on
	the manifold. (\cref{fig:tangentProj}).

	\begin{definition} The maximal solution in time of the \emph{reduced} dynamical
		system on $\M$,
		\begin{equation} \label{eqn:DOsystemAbstract}
			\left\{\begin{array}{cl}\dot{R}&=\Pi_\TT{R}(\LL(t,R))\\ R(0) &=
			\Pi_\M(\RR(0)),\end{array}\right.
		\end{equation} is called the \emph{Dynamically Orthogonal} (DO) approximation of
		\cref{eqn:dRstar}. The solution $R(t)=U(t)Z^T(t)$ is governed by a dynamical
		system for the mode matrix $U$ and the coefficient matrix $Z$ such that
		$(\dot{U},\dot{Z})\in \mathcal{H}_{(U,Z)}$ satisfies the \emph{dynamically
		orthogonal condition} $U^T\dot{U}=0$ at every instant:
		\begin{equation}\label{eqn:DOsystem} \left\{\begin{array}{rl} \dot{U}= &
				(I-UU^T)\mathcal{L}(t,UZ^T)Z(Z^TZ)^{-1} \\ \dot{Z}= & \mathcal{L}(t,UZ^T)^T
		U\\ U(0)Z(0)^T=& \Pi_\M(\RR(0)).\end{array}\right. \end{equation}
	\end{definition}
	\begin{remark}
		Equations \cref{eqn:DOsystem} are exactly those presented as DO equations in
		\cite{sapsis2009dynamically,sapsis2011dynamically}. With the notation of
		\cref{eqn:SPDE,eqn:KLdecomposition},  using $\langle \cdotp,\cdotp \rangle$ to
		denote the continuous dot product operator (an integral over the spatial domain)
		and $\mathbb{E}$ the expectation, they were written as the following set of
		coupled stochastic PDEs:
		\begin{equation}
			\label{eqn:DOsapsis}
			\left\{\begin{aligned} \partial_t \zeta_i & = \langle {\mathscr{L}(t,{\bm
							u}_\textrm{DO}^{\textrm{}};\omega)},{\bm u}_i \rangle\\ \sum_{j=1}^r
							\mathbb{E}[\zeta_i \zeta_j]\partial_t {\bm u}_j &
							=\mathbb{E}\left[\zeta_i\left({\mathscr{L}(t,{\bm
										u}_\textrm{DO}^{\textrm{}};\omega)}-\sum_{j=1}^r \langle
							{\mathscr{L}(t,{\bm u}_\textrm{DO}^{\textrm{}};\omega)},{\bm u}_j
				\rangle{\bm u}_j\right)\right]\;.
			\end{aligned}\right.
		\end{equation}
		However, when dealing with infinite dimensional Hilbert spaces, the vector space
		of solutions of \cref{eqn:SPDE} depends on the PDEs, which complicates the
		derivation of a general theory for \cref{eqn:DOsapsis}. Considering the DO
		approximation as a computational method for evolving low rank matrices relaxes
		these issues through the finite-dimensional setting.
	\end{remark}
	\begin{remark}
		One can relate \cref{eqn:DOsystemAbstract} to projected dynamical systems
		encountered in optimization \cite{Nagurney2012}, where the manifold $\M$ is
		replaced with a compact convex set.
	\end{remark}
	In the following, two justifications of the accuracy of this approximation are
	given. First, the DO approximation is shown to be the continuous limit of a scheme
	that would truncate the SVD of the full matrix solution after each time step, and hence is
	instantaneously optimal among any other possible model order reduced system. Then,
	its dynamics is compared to that of the best low rank approximation, yielding
	error bounds on global integration times. The efficiency of the DO approach in the
	context of the discretization of a stochastic PDE is not discussed here. These
	points are examined in \cite{Feppon2016a} and in references cited therein.  

	\subsection{The DO system applies instantaneously the truncated SVD} This paragraph
	details first a ``computational'' interpretation of the DO approximation.  Consider
	the temporal integration of the dynamical system \cref{eqn:dRstar} over $(t^n,
	t^{n+1})$, \begin{equation} \label{eqn:fullEulerScheme} \RR^{n+1}=\RR^n+\Delta t
	\overline{\mathcal{L}}(t^n,\RR^n,\Delta t),\end{equation} where
	$\overline{\mathcal{L}}(t,\RR,\Delta t)$ denotes the full-space integral
	$\overline{\mathcal{L}}(t,\RR,\Delta t)=\frac{1}{\Delta t}\int_t^{t+\Delta t}
	\LL(s,\RR(s))\D s$ for the exact integration or the increment function
	\cite{Haier2000} for a numerical integration.  Examples of the latter include
	$\overline{\mathcal{L}}(t,\RR,\Delta t)=\LL(t,\RR)$ for forward Euler and
	$\overline{\mathcal{L}}(t,\RR,\Delta t)=\LL(t+\Delta t/2,\RR+\Delta t/2\,
	\LL(t,\RR))$ for a second-order Runge-Kutta scheme.  Assume that the solution
	$\RR^n$ at time $t^n$ is well approximated by a rank $r$ matrix $R^n$. A natural
	way to  estimate the best rank $r$ approximation $\Pi_\M(\RR^{n+1})$ at the next
	time step is then to set
	\begin{equation} \label{eqn:fullSVDscheme}
		\left\{\begin{aligned}
			R^{n+1} & =\Pi_\M(R^n+\Delta t\overline{\LL}(t,R^n,\Delta t))\\
			R^0 & = \Pi_\M(\RR(0)).
		\end{aligned}\right.\end{equation}
		Such a numerical scheme uses the truncated SVD, $\Pi_\M$, to remove after each time
		step of the initial time-integration \cref{eqn:fullEulerScheme} the optimal amount
		of information required to constrain the rank of the solution. A data-driven
		adaptive version of this approach was for example used in
		\cite{Lermusiaux1997,lermusiaux_DAO1999}. 
		One can then look for a dynamical system for which
		\cref{eqn:fullSVDscheme} would be a temporal discretization. One then finds that,
		for any rank $r$ matrix $R\in\M$,
		\begin{equation}\label{eqn:consistency}\frac{\Pi_{\M}(R+\Delta
			t\overline{\mathcal{L}}(t,R,\Delta t))-R}{\Delta t}\underset{\Delta t\rightarrow
			0}\longrightarrow \DD_{\overline{\mathcal{L}}(t,R,0)}\Pi_\M(R)=
		\Pi_\TT{R}(\mathcal{L}(t,R))\end{equation} holds true since the curvature term
		depending on $N=R-\Pi_\M(R)=0$ vanishes in \cref{eqn:diffProjection}, and
		$\overline{\LL}(t,R,0)=\LL(t,R)$ by consistency of the time marching with the exact
		integration \cref{eqn:fullEulerScheme} \cite{Haier2000}. This implies, under
		sufficient regularity condition on $\LL$, that the continuous limit of the scheme
		\cref{eqn:fullSVDscheme} is the DO dynamical system \cref{eqn:DOsystemAbstract}.

	\begin{theorem}
		\label{prop:convergenceScheme}
		Assume that the DO solution \cref{eqn:DOsystemAbstract} is defined on a time
		interval $[0,T]$ discretized with $N_T$ time steps $\Delta t=T/N_T$ and denote
		$t^n=n\Delta t$. Consider $R^n$ the sequence obtained from the class of schemes
		\cref{eqn:fullSVDscheme}.  Assume that $\mathcal{L}$ is Lipschitz continuous,
		that is there exists a constant $K$ such that
		\begin{equation}\label{eqn:lipschitz} \forall t\in [0,T], ~\forall A,B\in \MlM, ~
		||\mathcal{L}(t,A)-\mathcal{L}(t,B)||\< K||A-B||.\end{equation} Then the sequence
		$R^{n}$ converges uniformly to the DO solution $R(t)$ in the following sense: \[
		\sup_{0\<n\<N_T} ||R^n-R(t^n)||\underset{ \Delta t\rightarrow 0}{\longrightarrow}
	0\]
	\end{theorem}
	\begin{proof}
		It is sufficient to check that the scheme \cref{eqn:fullSVDscheme} is both
		consistent and stable (see \cite{Haier2000}). Denote $\Phi$ the increment
		function of the scheme \cref{eqn:fullSVDscheme}:
		\begin{equation}\label{eqn:incrFunction} \Phi(t,R,\Delta t)=\frac{\Pi_\M(R+\Delta
		t\overline{\LL}(t,R,\Delta t))-R}{\Delta t}=\frac{1}{\Delta t}\int_0^1
	\frac{\D}{\D \tau}\Pi_\M(g(R,t,\tau,\Delta t))\D \tau\end{equation} with
	$g(R,t,\tau,\Delta t)=R+\tau \Delta t\overline{\LL}(t,R,\Delta t)$. Consider a
	compact neighborhood $\mathcal{U}$ of $\MlM$ containing the trajectory $R(t)$ on
	the interval $[0,T]$ and sufficiently thin such that $\mathcal{U}$ does not
	intersect the skeleton of $\M$. In particular, $\Pi_\M$ is differentiable with
	respect to $R$ on the compact neighborhood $\mathcal{U}$, hence Lipschitz
	continuous. The consistency of \cref{eqn:fullSVDscheme} and continuity of $\Phi$ on
	$[0,T]\times \mathcal{U}\times \R$ follows from \cref{eqn:consistency}. For usual
	time marching schemes (e.g.\;Runge Kutta), the Lipschitz condition
	\cref{eqn:lipschitz} also holds for the map $R\mapsto \overline{\LL}(t,R,\Delta
	t)$. Therefore it  $\Phi$ is also Lipschitz continuous with respect to $R$ on
	$\mathcal{U}$ by composition. This is a sufficient stability condition.
	\end{proof}

	As such, the DO approximation can be interpreted as the dynamical system that
	applies instantaneously the truncated SVD to constrain the rank of the solution.
	Therefore, other reduced order models of the form \cref{eqn:dRmanifold} are
	characterized by larger errors on short integration times for solutions whose
	initial value lies on $\M$.
	\begin{remark}
		Other dynamical systems that perform instantaneous matrix operations have been
		derived in \cite{brockett1988,Smith1991}, and in \cite{dehaene1995continuous}
		(e.g.\;Lemma\;3.4 and Corollary\;3.5) or \cite{dieci1999smooth} (sections 2.1 and
		2.3.) for tracking the full SVD or QR decomposition. Continuous SVD has been
		combined with adaptive Kalman filtering in uncertainty quantification to
		continuously adapt the dominant subspace supporting the stochastic solution
		\cite{Lermusiaux1997,lermusiaux_DAO1999,lermusiaux2001evolving}. All of these
		results utilized the instantaneous truncated SVD concept and formed the
		computational basis of the continuous DO dynamical system. In fact, the dominant
		singular vectors of state transition matrices and other operators have found
		varied applications in atmospheric and ocean sciences for some time
		\cite{farrell1996generalized_partI,farrell1996generalized_partII,palmer1998singular,hoskins2000nature,lermusiaux_PhysD2007,moore2004comprehensive,kalnay2003atmospheric,Diaconescu_laprise_SV_review_ESR2012}.
	\end{remark}

	\subsection{The DO approximation is close to the dynamics of the best low rank
	approximation of the original solution}
	Ideally, a model order reduced solution $R(t)$ would coincide at all times with the
	best rank $r$ approximation $\Pi_\M(\RR(t))$, so as to keep the error
	$||\RR(t)-R(t)||$ minimal. However, $\Pi_\M(\RR(t))$ is not
	the solution of a reduced system of the form \cref{eqn:dRmanifold} as its time
	derivative depends on the knowledge of the true solution $\RR$ in the full space
	$\mathcal{M}_{l,m}$. Indeed, formula \cref{eqn:diffSVDexplicit} for the differential
	of the SVD yields the following system of ODEs for
	the evolution of modes and coefficients of the best
	rank$-r$ approximation $\Pi_\M(\RR(t))$:
	\begin{equation}
		\label{eqn:SVDtracking}
		\left\{\begin{aligned}
				\dot{U}
				&=(I-UU^T)\dot{\RR}Z(Z^TZ)^{-1} \\
				& \qquad +\left[\sum_{\substack{1\<i\<r\\1\<j\<k}}
				\frac{\sigma_{r+j}}{\sigma_i^2-\sigma_{r+j}^2} (\sigma_i u_{r+j}^T\dot{\RR}
			v_i+\sigma_{r+j} u_i^T\dot{\RR} v_{r+j})u_{r+j}v_i^T\right]Z(Z^TZ)^{-1}\\
			\dot{Z} & =\dot{\RR}^TU+\left[\sum_{\substack{1\<i\<r\\1\<j\<k}}
				\frac{\sigma_{r+j}}{\sigma_i^2-\sigma_{r+j}^2} (\sigma_{r+j} u_{r+j}^T\dot{\RR}
			v_i+\sigma_{i} u_i^T\dot{\RR} v_{r+j})v_{r+j}u_i^T\right]U,
		\end{aligned}
	\right.
\end{equation}
where the (time-dependent) SVD of $\RR(t)$ at the time $t$ is
$\sum_{i=1}^{r+k}\sigma_i u_iv_i^T$ with $k=\min(m,l)$ (allowing possibly
$\sigma_{r+j}=0$ for $1\<j\<k$).  One therefore sees from this best rank$-r$
governing differential \cref{eqn:SVDtracking} that its reduced DO system
\cref{eqn:DOsystemAbstract} is obtained by (i) replacing the derivative
$\dot{\RR}=\LL(t,\RR)$ with the approximation $\LL(t,R)$ (first terms in each of the
right-hand sides of \cref{eqn:SVDtracking}), and (ii) neglecting the dynamics
corresponding to the interactions between the low-rank$-r$ approximation (singular
values and vectors of order $1\<i\<r$) and the neglected normal component (singular
values and vectors of order $r+j$ for $1\<j\<k$).  These interactions are the last
summation terms in each right-hand sides of \cref{eqn:SVDtracking}. Estimating these
interactions in all generality would require, in addition to the knowledge of a rank
$r$ approximation $R\simeq \Pi_\M(\RR)$, either external observations
\cite{lermusiaux_DAO1999} or closure models \cite{Wang_closurePOD_CMAME2012}, so as
to estimate the otherwise neglected normal component $\RR-\Pi_\M(\RR)=\sum_{j=1}^k
\sigma_{r+j}u_{r+j}v_{r+j}^T$.

Comparing the dynamics \cref{eqn:DOsystem} of the DO approximation to that of the
governing differential \cref{eqn:SVDtracking} of the best low rank$-r$ approximation,
a bound for the growth of the DO error is now obtained. 

	\begin{theorem}
		\label{thm:DOError}
		Assume that both the original solution $\RR(t)\in\MlM$
		(eqn.\;$\cref{eqn:dRstar}$) and its DO approximation $R(t)$
		(eqn.\;$\cref{eqn:DOsystemAbstract}$) are defined on a time interval $[0,T]$ and
		that the following conditions hold:
		\begin{enumerate}
			\item
			\label{eqn:cond1} $\mathcal{L}$ is Lipschitz continuous, \emph{i.e.}\;equation \cref{eqn:lipschitz} holds.
			\item \label{eqn:cond2} The original (true) solution $\RR(t)$ 
				remains close to the low rank manifold $\M$, in the sense that
				$\RR(t)$ does not cross the skeleton of $\M$ on $[0,T]$, \emph{i.e.}\;there is no crossing of the singular value of order $r$: \[\forall
				t\in[0,T],~\sigma_r(\RR(t))>\sigma_{r+1}(\RR(t)).\]
		\end{enumerate}
		Then, the error of the DO approximation $R(t)$ (eqn.\;\cref{eqn:DOsystemAbstract})
		remains controlled by the best approximation error $||\RR-\Pi_{\M}(\RR(t))||$ on
		$[0,T]$:
		\begin{multline}\label{eqn:DOBound}\forall t\in [0,T], ~ ||R(t)-\Pi_\M(\RR(t))||
		\< \\ \int_0^t ||\RR(s)-\Pi_\M(\RR(s))||\left(
	K+\frac{||\LL(s,\RR(s))||}{\sigma_{r}(\RR(s))-\sigma_{r+1}(\RR(s))}\right)e^{\eta
(t-s)}\D s ,\end{multline}
		where $\eta$ is the constant
		\begin{equation}\label{eqn:growthRate} \eta=
			K+\sup_{t\in[0,T]}\frac{2}{\sigma_r(\RR(t))}||\LL(t,\RR(t))||.\end{equation}
	\end{theorem}
	\begin{proof}
		A proof is given in \cref{app:proofError}.
	\end{proof}
	This statement improves the result expressed in \cite{koch2007dynamical} (Theorem
	5.1), since no assumption is made on the smallness of the best approximation error
	$||\RR-\Pi_\M(\RR)||$, nor on the boundedness of $||R-\Pi_\M(\RR)||$.
	\cref{thm:DOError} also highlights two sufficient conditions for the error
	committed by the DO approximation to remain small~: \subparagraph{Condition
	\ref{eqn:cond1}}  The discrete operator $\mathcal{L}$  must not be too sensitive to
	the error $\RR(t)-R(t)$, namely the Lipschitz constant $K$ must be small. This
	error is commonly encountered by any approximation made for evaluating the operator
	of a dynamical system (as a consequence of Gronwall's lemma \cite{hartman2002}).
	The Lipschitz constant $K$ also quantifies how fast the vector field $\mathcal{L}$
	may deviate from its values when getting away from the low rank manifold $\M$.
	\subparagraph{Condition \ref{eqn:cond2}}  Independently of the choice of the
	reduced order model, the solution of the initial system \cref{eqn:dRstar},
	$\RR(t)$, must remain close to the manifold $\M$, or in other words, must remain
	far from the skeleton $\Sk(\M)$ of $\M$. As visible on \cref{fig:parabola}, the
	best rank $r$ approximation $\Pi_\M(\RR)$ of $\RR$ exhibits a jump when $\RR$
	crosses the skeleton, \emph{i.e.}\;when $\sigma_r(\RR)=\sigma_{r+1}(\RR)$ occurs.
	At that point, the discontinuity of $\Pi_\M(\RR(t))$ cannot be tracked by the DO or
	any other smooth dynamical approximation.  Condition \ref{eqn:cond2} in some sense
	supersedes the stronger condition of ``smallness of the initial truncation error''
	of the error analysis of \cite{koch2007dynamical}.  Indeed, when
	$\sigma_r(\RR)\simeq\sigma_{r+1}(\RR)$ occurs, as observed numerically in
	\cite{musharbash2015error}, the DO solution may then diverge sharply from the SVD
	truncation. From the point of view of model order reduction, the resulting error
	can be related to the evolution of the residual $\RR-\Pi_\M(\RR)$ that is not
	accounted for by the reduced order model. When the crossing of singular values
	occurs, neglected modes in the approximation \cref{eqn:KLdecomposition} become
	``dominant", but cannot be captured by a reduced order model that has evolved only
	the first modes initially dominant. In such cases, one has to restart the
	simulations from the initial conditions with a larger subspace size or the size of
	the DO subspace has to be increased and corrections applied from external
	information. The latter learning of the subspace can be done from measurements or
	from additional Monte-Carlo simulations and breeding of the best low-rank$-r$
	approximation \cite{lermusiaux_DAO1999,kalnay2003atmospheric,sapsis2012dynamical}.

	Last, it should be noted that the growth rate $\eta$ (equation
	\cref{eqn:growthRate}) of  the error increases as the evolved trajectory becomes
	close to be singular, \emph{i.e.}\;when $\sigma_r(\RR(t))$ goes to zero. This
	growth rate comes mathematically from the Gronwall estimates of the proofs, and  is
	intuitively related to the fact the tangent projection $\Pi_\mathcal{T}$ in
	\cref{eqn:DOsystemAbstract} is applied at the location of the DO solution $R(t)$
	instead of the one of the best approximation $\Pi_\M(\RR(t))$. If the evolved
	trajectory is close to be singular, the local curvature of $\M$ experienced by the
	DO solution $R(t)$ and the best approximation $\Pi_\M(\RR(t))$ is high. Therefore
	the tangent spaces $\TT{R(t)}$ and $\TT{\Pi_\M(\RR(t))}$ may be oriented very
	differently because of this curvature, resulting in increased error when
	approximating the tangent projection operator $\Pi_\TT{\Pi_\M(\RR(t))}$ by
	$\Pi_\TT{R(t)}$ in the DO system \cref{eqn:DOsystemAbstract}.
	\begin{remark}
		\Cref{prop:convergenceScheme,thm:DOError} may be generalized in a straightforward
		manner to the case of any smooth embedded manifolds $\M\subset E$ (Theorem
		2.5 and 2.6 in \cite{FepponThesis}).
	\end{remark}
	\section{Optimization on the fixed rank matrix manifold for tracking the best low rank
	approximation}
	\label{sec:numerical}
	This section applies the framework of Riemannian matrix optimization
	\cite{edelman1998,absil2009all} as an alternative approach to the direct tracking
	of the truncated SVD of a time-dependent matrix $\RR(t)\in\MlM$. At the end, we
	provide a remark (\cref{remark:DO-optimization}) linking the two approaches within
	the context of the DO system.

	Consider a given (full-rank) matrix $\RR\in\MlM$ and recall that $\Pi_\M(\RR)$,
	when it is non-ambiguously defined, is the unique minimizer of the distance
	functional
	\begin{equation} \label{eqn:J}\begin{array}{ccccc}J &: &
		\M & \longrightarrow & \R\\ & & R & \longmapsto & \frac{1}{2}||R-\RR||^2 .
	\end{array}\end{equation}
		Riemannian optimization algorithms, namely gradient descents and Newton
		methods on the fixed rank manifold $\M$, are now used to provide alternative ways
		to more standard direct algebraic algorithms \cite{Golub2012} for
		evaluating the truncated SVD $\Pi_\M(\RR)$. Such optimizations can be useful to
		dynamically update the best low rank approximation of a time dependent matrix
		$\RR(t)$: this is because for a sufficiently small time step $\Delta t$,
		$R(t)=\Pi_\M(\RR(t))$ is expected to be close to $R(t+\Delta
		t)=\Pi_\M(\RR(t+\Delta t))$,  hence $\Pi_\M(\RR(t))$ provides a good initial
		guess for the minimization of $R\mapsto ||\RR(t+\Delta t)-R||$. The minimization
		of the distance functional $J$ has already been considered in the matrix
		optimization community
		\cite{absil2009optimization,vandereycken2013low,mishra2014} that derived gradient
		descent and Newton methods on the fixed rank manifold, but not in the case of the
		metric inherited from the ambient space $\MlM$ (eqn.\;\cref{eqn:metric}), which
		is done in what follows. As a benefit of this ``extrinsic'' approach already
		noticed in \cite{absil2013extrinsic}, the covariant Hessian of $J$ relates
		directly to the Weingarten map at critical points: this will allow obtaining the
		convergence of the gradient descent for almost every initial data
		(\cref{prop:localMinima}).

	Ingredients required for the  minimization of $J$ on the manifold $\M$ are first
	derived, namely the covariant gradient and Hessian. As reviewed in
	\cite{edelman1998}, usual optimization algorithms such as gradient and Newton
	methods can be straightforwardly adapted to matrix manifolds.
	The differences with their Euclidean counterparts is that: (i) usual gradient and
	Hessians must be replaced by their covariant equivalents; (ii) one needs to follow
	geodesics instead of straight lines to move on the manifold; and, (iii) directions
	followed at the previous time steps, needed for example in the conjugate gradient
	method, must be  transported to the current location (equation
	\cref{eqn:parallelTransport}). Covariant gradient and Hessian are recalled in the
	following definition (for details, see \cite{absil2009optimization}, chapter 5).
	\begin{definition}
		\label{def:gradienthessian}
		Let $J$ be a smooth function defined on $\M$ and $R\in\M$. The covariant gradient
		of $J$ at $R$ is the unique vector $\nabla J\in \TT{R}$ such that \[ \forall X\in
		\TT{R},\, J(\exp_R(tX))=J(R)+t\langle \nabla J,X \rangle+o(t).\] The covariant
		Hessian $\mathcal{H}J$ of $J$ at $R$ is  the linear map on $\TT{R}$ defined by \[
		\mathcal{H}J(X)=\nabla_X \nabla J, \] and the following second order Taylor
		approximation of $J$ holds: \[ J(\exp_R(tX))=J(R)+t\langle \nabla J,X
		\rangle+\frac{t^2}{2} \langle X,\mathcal{H}J(X) \rangle+o(t^2).\]
	\end{definition} The following proposition (see \cite{absil2013extrinsic}) explains
	how these quantities are related to the usual gradient and Hessian, so that they
	become accessible for computations.
	\begin{proposition} \label{prop:gradientHessian}
		Let $J$ be a smooth function defined in the ambient space $\MlM$ and denote $\DD
		J$  and $\DD^2 J$ its respective Euclidean gradient and Hessian. Then the
		covariant gradient and Hessian are given by
		\begin{equation} \label{eqn:covgradient} \nabla J=\Pi_\TT{R}(\DD
		J),\end{equation} \begin{equation}\label{eqn:covhessian}
		\mathcal{H}J(X)=\Pi_\TT{R}(\DD^2 J(X))+\DD\Pi_\TT{R}(X)\cdot
	\left[(I-\Pi_\TT{R})(\DD J)\right]. \end{equation}
	\end{proposition}
Applying directly  \cref{prop:gradientHessian}, the gradient and the Hessian of $J$
	at $R=UZ^T\in \M$ are given by:
	\begin{equation}\label{eqn:covgradientJ}\nabla J
	=((I-UU^T)(UZ^T-\RR)Z(Z^TZ)^{-1},(UZ^T-\RR)^TU), \end{equation}
\begin{equation}\label{eqn:covhessianJ} \begin{array}{ccccc} \mathcal{H}J &: &
\mathcal{H}_{(U,Z)} & \rightarrow & \mathcal{H}_{(U , Z)} \\ &  &
\begin{pmatrix}X_U\\X_Z\end{pmatrix} & \mapsto  &
\begin{pmatrix}X_U-N_{UZ^T}(\RR)X_Z(Z^TZ)^{-1}\\X_Z-N_{UZ^T}(\RR)^TX_U\end{pmatrix},\end{array}
\end{equation}
	where $N_{UZ^T}(\RR)=(I-\Pi_\TT{UZ^T})(\RR-UZ^T)=(I-UU^T)\RR(I-Z(Z^TZ)^{-1}Z^T)$ is
the orthogonal projection of $\RR-R$ onto the normal space. The Newton direction $X$
is found by solving the linear system $\mathcal{H}J(X)=-\nabla J(R)$, that reduces to
\[ \left\{\begin{array}{r} X_UA +BX_Z  =E\\ B^TX_U+X_Z = F,\end{array}\right.\] with
	$A=(Z^TZ)$, $B=-N_{UZ^T}(\RR)$, $E=(I-UU^T)\RR Z$ and $F=-Z+\RR^TU$. This requires
	to solve the Sylvester equation $X_UA-BB^TX_U=E-BF$ for $X_U$, that can be  done in
	theory by using standard techniques \cite{kirrinnis2001fast}, before computing
	$X_Z$ from $X_Z=F-B^TX_U$.

	It is now proven that the distance function $J$ may admit several critical points,
	but a unique local, hence global, minimum  on $\M$. 	As a consequence, saddle
	points of $J$ are unstable equilibrium solutions of the gradient flow
	$\dot{R}=-\nabla J(R)$ and hence are expected to be avoided by gradient descent,
	which will converge in practice to the global minimum $\Pi_\M(\RR)$. 
	This ``almost surely'' convergence guarantee for the gradient
	descent may be compared to probabilistic analyses 
	investigated in more general contexts \cite{Pitaval2015,WeiCaiChanEtAl2016}.
  Our result also shows that one cannot expect 
	the Newton method to converge for initial guesses that are far from the optimal.
	Indeed, this method seeks a zero of the gradient $\nabla J$ rather
	than a true minimum, and hence may converge or oscillate around several of the
	saddle points of the objective function.
	\begin{proposition}
		\label{prop:localMinima}
		Consider $\RR\in\MlM$ such that its projection onto $\M$ is well defined, that is
		$\sigma_r(\RR)>\sigma_{r+1}(\RR)$. Then the distance function $J$ to $\RR$
		(eqn.\;\cref{eqn:J}) admits no other local minima than $\Pi_{\M}(\RR)$. In other
		words, for almost any initial rank $r$ matrix $U(0)Z(0)^T$, the solution
		$U(t)Z(t)^T$ of the gradient flow 
		\begin{equation}
			\left\{\begin{aligned}
 \dot{U} & =(I-UU^T)\RR Z(Z^TZ)^{1}\\
 \dot{Z} & =\RR^TU-Z
				\end{aligned}\right.
		\end{equation}
		converges to $\Pi_\M(\RR)$, the rank $r$ truncated SVD of $\RR$.
	\end{proposition}
	\begin{proof}
		It is
		known from \cref{prop:normal} that the points $R$ for which $\nabla J$
		vanishes are such that  $\DD J=R-\RR\in\NN{R}$ is a normal vector.
		Since in addition $\DD^2 J=I$, \cref{prop:gradientHessian} yields  the identity
		\[\begin{aligned}
				\forall X\in\TT{R}, \quad \langle \mathcal{H} J(X),X \rangle & =
			\langle X, X\rangle -\langle\DD\Pi_\TT{R}(X)\cdot N,X\rangle \\
			& = 			||X||^2-\langle L_R(N)(X),X
\rangle, \end{aligned}\]  where $N=-(I-\Pi_\TT{R})(\DD J)=-\DD J=\RR-R\in \NN{R}$, since $\nabla J=\Pi_\TT{R}(\DD
J)$ vanishes at $R$.
  Let $\RR=\sum_{i=1}^{r+k} \sigma_i(\RR) u_iv_i^T$ be the SVD of $\RR$.
	For $\RR-R$ to be a normal vector, $R$ must necessary be of the form
	$R=\sum_{i\in A} \sigma_i u_iv_i^T$ where $A$ is a
		subset of $r$ indices $1\<i\<r+k$. Then the minimum
		eigenvalue of the Hessian $\mathcal{H}$ is $1-\frac{\sigma_1(N)}{\sigma_r(R)}$,
		which is positive if and only if $\sigma_r(R)>\sigma_1(N)$. This happens only for
		$R=\Pi_\M(\RR)$.
	\end{proof}
	\begin{remark}
		The reader is referred to \cite{jost2008riemannian} for details regarding the
		convergence almost surely of sufficiently smooth gradient flows towards the
		unique minimizer of a function (Morse theory).
	\end{remark}
	 On \cref{fig:optimization}, a matrix $\RR\in\MlM$ with
	$m=100$ and $l=150$ is considered, with singular values chosen to be equally spaced
	in the interval $[1,10]$. Three optimization algorithms detailed in
	\cite{edelman1998} (gradient descent with fixed step, conjugate gradient descent,
	and Newton method) are implemented to find the best rank $r=5$  approximation of
	$\RR$, with a random initialization. Convergence curves are plotted on
	\cref{fig:optimization}: linear and quadratic rates characteristic of respectively
	gradient and Newton methods are obtained. As expected from \cref{prop:localMinima},
	gradient descents globally converge to the truncated SVD, while Newton iterations
	may be attracted to any saddle point.
	\begin{figure} \centering
		\includegraphics[height=3.47cm]{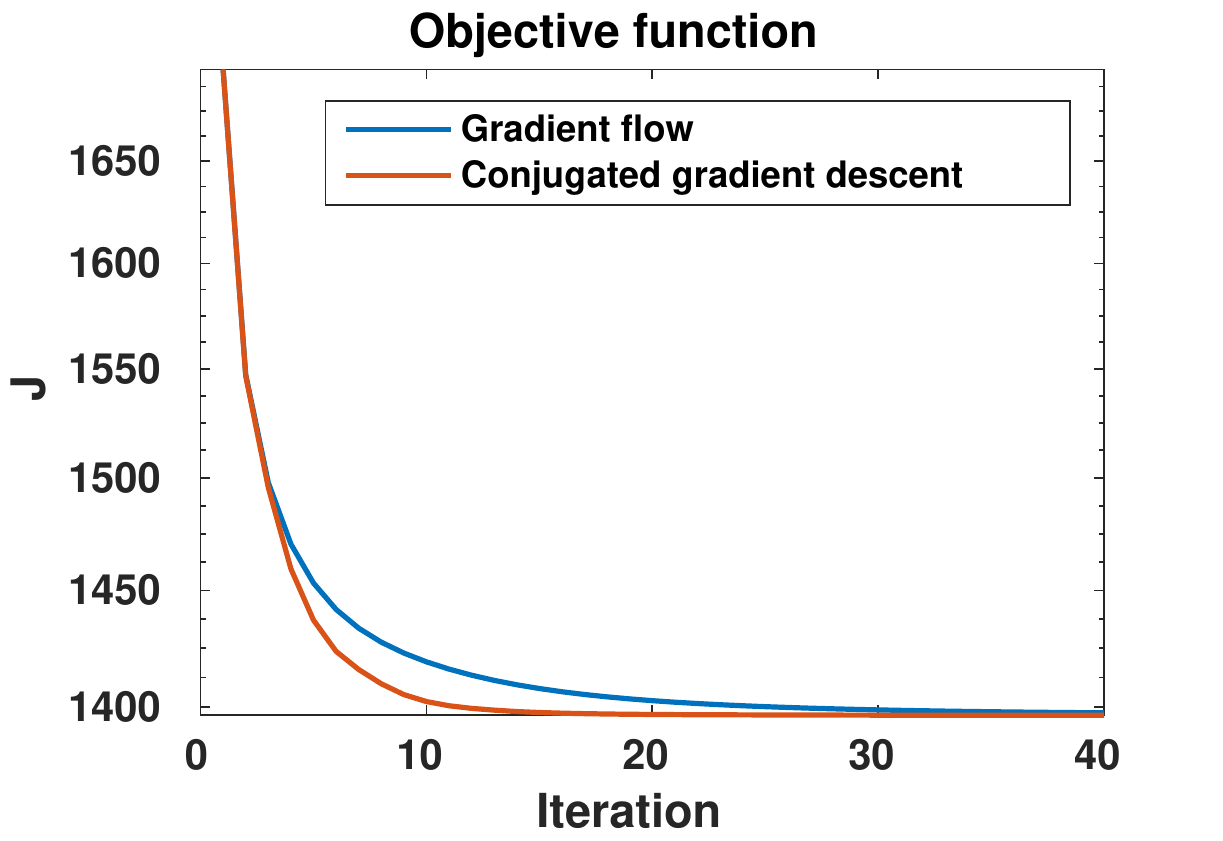}
		\includegraphics[height=3.47cm]{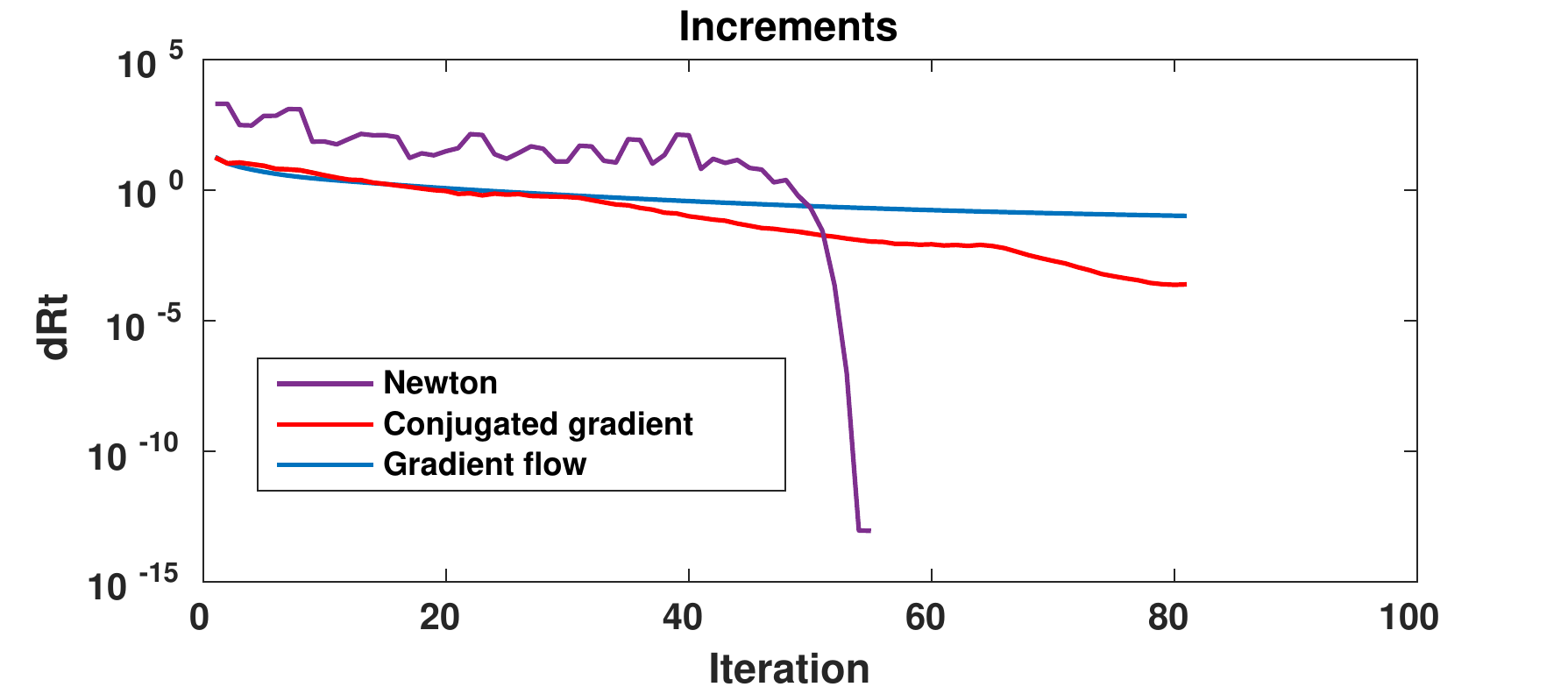}
		\caption{\small Convergence curves of optimization algorithms for minimizing the
	distance function $J$ (equation \cref{eqn:J}). Newton does not converge to the
	global minimum and hence is not represented on the left curve.}
    \label{fig:optimization} 
\end{figure}

\begin{remark} \label{remark:DO-optimization}
The above gradient descent and Newton methods can be combined with previously-derived
numerical schemes for the time-integrated DO eqs.\;\cref{eqn:fullSVDscheme}.  One
class of schemes consists of discretizing the ODEs \cref{eqn:DOsystem} in time, as in
\cite{sapsis2009dynamically,ueckermann_et_al_JCP2013,musharbash2015error,koch2007dynamical}.
Another follows \cref{eqn:fullSVDscheme} directly and aims to compute the SVD
truncation $\Pi_\M(\RR)$ of $\RR=UZ^T+\Delta t\,\overline{\LL}(t,UZ^T,\Delta t)$,
where the increment function can be that of Euler or of higher-order explicit time
marching (of course, the total rank of this $\RR$ depends on the dynamics and
numerical scheme, and can be greater than $r$). 
Examining the expression of the gradient of $J$ (eqn.\;\cref{eqn:covgradientJ}), one
time-step of the above schemes can be interpreted as one gradient descent step for
minimizing the functional $J$. Therefore, optimization algorithms on the Riemannian
manifold $\M$ can be combined with such DO time-stepping schemes, as further
investigated in \cite{Feppon2016a}. A key advantage of
such optimization is the capability of altering the rank $r$ of the dynamical
approximation over a time step or stage (e.g.\;a rank $p>r$ approximation can be used
in the target cost functional $J$). These strategies may also be utilized for the
computation of nonlinear singular vectors \cite{VaidyaNagarajothers2010} or for
continuous dominant subspace estimation \cite{lermusiaux2001evolving}. Finally, it
can also be combined with adaptive learning schemes
\cite{lermusiaux_DAO1999,lermusiaux_PhysD2007,sapsis2012dynamical} which use system
measurements and/or Monte-Carlo breeding nonlinear simulations to estimate the
missing fastest growing modes. Such additional information can then correct the
predictor of the SVD of $\RR(t+\Delta t)$ in directions orthogonal to the discrete DO
increments and essentially increase the subspace size, e.g.\;when the estimates of
$\sigma_{r+1}(\RR(t))$ become close to these of $\sigma_r(\RR(t))$.
\end{remark}

	\section{Conclusion}

	A geometric approach was developed for dynamical model-order reduction, through the
	analysis of the embedded geometry of the fixed rank manifold $\M$. The extrinsic
	curvatures of matrix manifolds were studied and geodesic equations obtained. The
	relationships among these notions and the differential of the orthogonal projection
	of the original system dynamics onto the tangent spaces of the manifold were
	derived and linked to the DO approximation. These geometric results allowed to
	derive the differential of the truncated SVD 
	interpreted as
	an 	orthogonal projection onto the fixed rank matrix manifold. The DO approximation, with its
	instantaneous application of the SVD truncation of the stochastic/parametric
	dynamics, was shown to be the natural dynamical reduced-order model that is optimal
	on small integration times among all other reduced-order models that evaluate the
	operator of the full-space dynamics exclusively onto low rank approximations.
	Additionally, the explicit dynamical system satisfied by the best low rank
	approximation was derived and used to sharpen the error analysis of the DO
	approximation.

 The DO method was related
	to Riemannian matrix optimization, for which gradient descent methods were applied
	and shown capable of adaptively tracking the best low rank approximation of dynamic
	matrices. This may prove beneficial in the integration of the time stepping of the
	DO approximation.  Such approaches, in
	contrast with classic numerical integrations of the governing differential
	equations for the DO modes and their coefficients, open new future avenues for efficient
	DO numerical schemes.
	In general, there are now many promising directions for
	developing new, efficient, dynamic reduced-order methods, based on the geometry and
	shape of the full-space dynamics. Opportunities abound over a wide range of needs
	and applications of uncertainty quantification and dynamical system analyses and
	optimization in oceanic and atmospheric sciences, thermal-fluid sciences and
	engineering, electrical engineering, and chemical and biological sciences and
	engineering.

	\section*{Acknowledgments} We thank the members of the MSEAS group at MIT as well
	as Camille Gillot, Christophe Zhang, and Saviz Mowlavi for insightful discussions
	related to this topic. We are grateful to the Office of Naval Research for support
	under grants N00014-14-1-0725
	(Bays-DA) and N00014-14-1-0476 (Science of Autonomy -- LEARNS) to the Massachusetts Institute of Technology.
	\appendix

	\section{Proof of \cref{thm:distanceDiff}}
	\label{app:proofDiff}
	\begin{lemma}
		\label{lem:continuityProj}
		Let $\Omega$ be an open set over which the projection $\Pi_\M$ is uniquely
		defined by eqn.\;\cref{eqn:CondunicityProj}
		and such that condition \cref{eqn:Condfrontier} holds.
		Then  $\Pi_\M$ is continuous on $\Omega$.
	\end{lemma}
	\begin{proof}
		Consider a sequence $\RR_n$ converging in $E$ to $\RR$ and denote $\Pi_\M(\RR_n)$
		the corresponding projections. Let $\epsilon>0$ be a real such that $\forall
		n\>0, ||\RR_n-\RR||<\epsilon$. Since \[\begin{aligned}||\Pi_\M(\RR_n)-\RR|| &
				\<||\Pi_\M(\RR_n)-\RR_n||+||\RR_n-\RR|| \\ &
		\<||\RR_n-\Pi_\M(\RR)||+||\RR_n-\RR||\\ & \<
2\epsilon+||\RR-\Pi_\M(\RR)||,\end{aligned}\] the sequence $\Pi_\M(\RR_n)$ is
bounded. Denote $R\in\overline{\M}$ a limit point of this sequence. Passing to the
limit the inequality $||\RR_n-\Pi_\M(\RR_n)||\<||\RR_n-\Pi_\M(\RR)||$, one obtains
$||\RR-R||\<||\RR-\Pi_\M(\RR)||$. The uniqueness of the projection, and the fact that
there is no $R\in\overline{\M}\backslash \M$ satisfying this inequality, shows that
$R=\Pi_\M(\RR)$. Since the bounded sequence $(\Pi_\M(\RR_n))$ has a unique limit
point, one deduces the convergence $\Pi_\M(\RR_n)\rightarrow \Pi_\M(\RR)$ and hence
the continuity of the projection map at $\RR$.
	\end{proof}
	\begin{lemma}
		\label{lem:kappaicond}
		At any point $\RR\in\Omega$, any principal curvature $\kappa_i(N)$ in the direction $N$ at $\Pi_\M(\RR)$ must satisfy $\kappa_i(N)<1$.
	\end{lemma}
	\begin{proof}
		It is shown in \cref{prop:gradientHessian}	that the covariant Hessian of the distance function $J(R)=\frac{1}{2}||\RR-J||^2$ at $R=\Pi_\M(\RR)$ is given by
		\begin{equation} \label{eqn:HessianJ}\begin{array}{cccccc}\mathcal{H}J &: & \TT{R} & \rightarrow & \TT{R} \\
				& 	& X & \mapsto & X-L_R(N)(X),\end{array}\end{equation}
		where $N$ is the normal direction $N=\RR-\Pi_\M(\RR)$.
		Since $R=\Pi_\M(\RR)$ must be a local minimum of $J$, this Hessian must be positive, namely any eigenvalue $\kappa_i(N)$ of the Weingarten map $L_R(N)$ must satisfy $1-\kappa_i(N)\>0$. Now, consider $s>1$ such that $R+sN\in\Omega$ and notice  that $||R+sN-\Pi_\M(\RR)||=s||N||$.
		Since \[||R+sN-\Pi_\M(R+sN)||\<||R+sN-\Pi_\M(\RR)||=s||N||,\]
		the uniqueness of the projection in $\Omega$ implies that $\Pi_\M(R+sN)=R$ (\emph{i.e.}\;the projection is invariant along orthogonal rays). The linearity of the Weingarten map in $N$ implies $\kappa_i(sN)=s\kappa_i(N)$,  hence $\kappa_i(N)\<\frac{1}{s}<1$, which concludes the proof.
	\end{proof}
	\begin{proof}[Proof of \cref{thm:distanceDiff}]
		Consider the function $f(\RR,R)=\Pi_\TT{R}(R-\RR)$ defined on $\M\times E$. The
		differential of $f$ with respect to the variable $R$ in a direction $X\in \TT{R}$
		at $R=\Pi_\M(\RR)$ is the application \[ X\mapsto \Pi_\TT{R}X-\DD_X
		\Pi_\TT{R}(\RR-R)=(I-L_R(N))(X).\] \cref{lem:kappaicond} implies that the
		Jacobian $\partial_{R,X} f$ has no zero eigenvalue and hence is invertible. The
		implicit function theorem ensures the existence of a diffeomorphism $\phi$
		mapping an  open neighborhood $\Omega_E\subset E$ of $\RR$ to an open
		neighborhood $\Omega_\M\subset \M$ of $R$, such that for any $\RR'\in\Omega_E$,
		$\phi(\RR')$ is the unique element of $\Omega_\M$ satisfying
		$f(\RR',\phi(\RR'))=0$. By continuity of the projection
		(\cref{lem:continuityProj}),  one can assume, by replacing $\Omega_E$ with the
		open subset $\Omega_E\cap \Pi_\M^{-1}(\Omega_\M)$, that $\Pi_\M(\Omega_E)\subset
		\Omega_\M$. Then, the equality $f(\RR',\Pi_\M(\RR'))=0$ implies by uniqueness:
		$\phi(\RR')=\Pi_\M(\RR')$. Hence $\Pi_\M=\phi$ on $\Omega_E$, and, in particular,
		$\Pi_\M$ is differentiable. Finally, for a given $X\in E$, one can now solve
		\cref{eqn:SVDdifferential1}
		by projection onto the eigenvectors of $L_R(N)$ and obtain
		\cref{eqn:diffProjection}.\,
	\end{proof}

	\section{Proof of \cref{thm:DOError}}
	\label{app:proofError}
	\begin{lemma}
		\label{corol:boundDXPiM}
		For any $\RR\in\MlM$ satisfying $\sigma_r(\RR)>\sigma_{r+1}(\RR)$ and
		$\XX\in\MlM$~: \[
		||\DD_\XX\Pi_\M(\RR)-\Pi_\TT{\Pi_\M(\RR)}\XX||\<\frac{\sigma_{r+1}(\RR)}{\sigma_{r+1}(\RR)-\sigma_r(\RR)}||\XX||.\]
	\end{lemma}
	\begin{proof}
		This is a consequence of the fact that the maximum eigenvalue in the
		decomposition \cref{eqn:diffProjection} is \[
		\max_{i,j}\frac{\frac{\sigma_{r+j}(\RR)}{\sigma_i(\RR)}}{1-\frac{\sigma_{r+j}(\RR)}{\sigma_i(\RR)}}=\frac{\sigma_{r+1}(\RR)}{\sigma_r(\RR)-\sigma_{r+1}(\RR)}.\]
	\end{proof} 
	The following lemma can be found in \cite{WeiCaiChanEtAl2016} and
	Theorem 2.6.1 in \cite{Golub2012}.  
	\begin{lemma} \label{lemma:curvatureLemma2} For
		any points $R^1, R^2\in\M$ the following estimate holds:
		\begin{equation}
			\label{eqn:curvatureBound}
		||\Pi_\TT{R^1}-\Pi_\TT{R^2}||\<\min\left(1,\frac{2}{\sigma_r(R^1)}||R^1-R^2||\right
),\end{equation}
	where the norm of the left-hand side is the operator norm.\end{lemma}
	\begin{remark}
		This result from \cite{WeiCaiChanEtAl2016} 
		enhances the ``curvature estimates'' of Lemma 4.2. of
		\cite{koch2007dynamical} that allows to 
		have a global bound and hence avoids the smallness assumption of the initial
		truncation error. Note that such a bound always exists at every points of smooth
		manifolds (Definition 2.17 of \cite{FepponThesis}). A purely geometric analysis 
		(Lemma 3.1. in \cite{FepponThesis}) may also be used to yield
		locally a sharper bound than \cref{eqn:curvatureBound} but with a larger constant
		$5/2$ instead of $2$ as a global estimate.
	\end{remark}
	\begin{proof}[Proof of \cref{thm:DOError}]
		Denote $R^*(t)=\Pi_\M(\RR(t))$.		Since
		$\dot{R}^*(t)=\DD_{\dot{\RR}}\Pi_\M(\RR(t))$, bounding
		\cref{eqn:SVDdifferential1} and using \cref{eqn:dRstar} and \cref{corol:boundDXPiM} yields: \[
		||\dot{R}-\dot{R}^*||\<||\Pi_\TT{R^*}(\mathcal{L}(t,\RR))-\Pi_\TT{R}(\mathcal{L}(t,R))||+\frac{\sigma_{r+1}(\RR)}{\sigma_r(\RR)-\sigma_{r+1}(\RR)}||\LL(t,\RR)||.\]
		Furthermore, by triangle inequality, \[ \begin{aligned}
				||\Pi_\TT{R^*}(\mathcal{L}(t,\RR))-\Pi_\TT{R}(\mathcal{L}(t,R))||   & \<
				||\Pi_\TT{R^*}(\LL(t,\RR))-\Pi_\TT{R}(\LL(t,\RR))||\\
				&+||\Pi_\TT{R}(\LL(t,\RR))-\Pi_\TT{R}(\LL(t,R^*))|| \\
		&+||\Pi_\TT{R}(\LL(t,R^*))-\Pi_\TT{R}(\LL(t,R))||.\end{aligned}\] The  \cref{lemma:curvatureLemma2} (first eqn.) and Lipschitz
		continuity of $\LL$ (last two eqs.) then imply
		\begin{gather*}
			||\Pi_\TT{R^*}(\LL(t,\RR))-\Pi_\TT{R}(\LL(t,\RR))||\<
			\frac{2}{\sigma_r(R^*)}||R-R^*||\,||\LL(t,\RR)||, \\
			||\Pi_\TT{R}(\LL(t,\RR))-\Pi_\TT{R}(\LL(t,R^*))||\< K||\RR-R^*||,\\
			||\Pi_\TT{R}(\LL(t,R^*))-\Pi_\TT{R}(\LL(t,R))||\<K||R-R^*||.
		\end{gather*}
		Finally, the following inequality is derived, combining all above equations together:
		\begin{multline}||\dot{R}-\dot{R}^*||\\ \<
		\left(K+\frac{2||\LL(t,\RR)||}{\sigma_r(R^*)}\right)||R-R^*||+\left(K+\frac{||\LL(t,\RR)||}{\sigma_r(\RR)-\sigma_{r+1}(\RR)}\right)||\RR-R^*||.\end{multline}
		An application of Gronwall's Lemma  (see corollary 4.3. in \cite{hartman2002})
		yields  \cref{eqn:DOBound}.\,
	\end{proof}

	\bibliographystyle{siamplain}
	\bibliography{extracted}
\end{document}